\documentclass{article} \voffset=-.5in
\usepackage{a4wide}
\usepackage{xcolor}
\usepackage{amsmath,amssymb,amsthm}
\usepackage[shortlabels]{enumitem}

\newtheorem{theorem}{Theorem}
\newtheorem{lemma}[theorem]{Lemma} 

\newtheorem{proposition}[theorem] {Proposition}
\theoremstyle{definition}
\newtheorem{example}[theorem]{Example}
\newtheorem*{definition}{Definition}
\newtheorem{remark}[theorem]{Remark}

\usepackage{tikz,tikz-cd}
\usepackage{float}
\usetikzlibrary{arrows}
\usepackage{mathtools}
\usepackage[all]{xy}

\DeclareMathOperator{\Fix}{Fix}

\newcommand{\Q}{\mathbb{Q}}
\newcommand{\Z}{\mathbb{Z}}

\newcommand{\R}{\mathbb{R}}

\newcommand{\nilk}[2]{\raisebox{-1.2mm}{\small$#1$\!}\backslash\raisebox{1.1mm}{$\!#2$}}
\newcommand{\lie}{{\mathfrak g}}

\numberwithin{theorem}{section}

\title{\bf Nielsen numbers of affine $n$-valued maps on nilmanifolds}

\author{C.\ Deconinck, \;\;
K.\ Dekimpe\thanks{Research supported by long term structural funding - Methusalem grant of the Flemish Government.}\\
\small {KU Leuven Campus Kulak Kortrijk},
\small {8500 Kortrijk},
\small {Belgium} \\
\small {e-mail: Charlotte.Deconinck@kuleuven.be},\;\;
\small {Karel.Dekimpe@kuleuven.be}}

\begin{document}

\maketitle

\begin{abstract}
A nilmanifold is a quotient $\nilk{N}{G}$ of a connected and simply connected nilpotent Lie group $G$ by a uniform lattice $N$.
In this paper we determine the Reidemeister and Nielsen number of affine $n$-valued maps on such a nilmanifold. These are maps for which a given lifting to $G$ splits into $n$ affine maps of the Lie group $G$. In order to obtain this result we also establish a way of computing the number of generalized twisted conjugacy classes on finitely generated torsion free nilpotent groups.
\end{abstract}

\section{Introduction}
Let $f:X\to X$ be a continuous map on a reasonably nice topological space $X$ (e.g.\ $X$ is a finite polyhedron). Nielsen fixed point theory (see \cite{jm}, \cite{ji}, \cite{j} and \cite{k}) is a part of algebraic topology that provides a framework to determine the value of 
\[ \text{MF}(f)= \min \#\{ {\rm Fix}(g) \mid g \simeq f \},\]
where ${\rm Fix}(g)=\{ x\in X\mid g(x)=x\}$ is the set of fixed points of $g$ and $g\simeq f$ denotes the fact that $g$ is homotopic to $f$. This is done by attaching two numbers to the map $f$. First one divides the set ${\rm Fix}(f)$ into so-called fixed point classes. This is done in such a way that for a map $g$ homotopic to $f$, the fixed point classes of $g$ are in one-to-one correspondence to those of $f$. In this way, the number of fixed point classes, called the Reidemeister number $R(f)$ of $f$, is invariant under homotopy. Note that some of the fixed point classes of $f$ (or of $g$) might be empty. In a second step, one attaches an integer to each fixed point class. This integer is the index of the fixed point class and a fixed point class is called essential if its index is non-zero. The idea is that an essential fixed point class can never become empty when considering homotopies, while a non-essential one can. The Nielsen number $N(f)$ of the map $f$ is then the number of essential fixed point classes of $f$. Also $N(f)$ is homotopy invariant. This Nielsen number is a very good estimate for ${\rm MF}(f)$, indeed, by a result of F.~Wecken we have that when  $X$ is a compact manifold of dimension different from two that $N(f)= {\rm MF}(f)$ for all self-maps of $X$ (see \cite{w}). 

In general $R(f)$ is fairly easy to compute using algebraic tools ($R(f)$ equals the number of twisted conjugacy classes of the induced endomorphism on the fundamental group $\pi(X)$ of $X$, see also Section~\ref{sec: R for groups}), while unfortunately $N(f)$ is very often hard to compute. However, for some classes of manifolds one has been able to develop easy formulas to compute this Nielsen number. This is for instance the case for nilmanifolds. We will give more details on this class of manifolds in Section~\ref{sectie3}, but for now it suffices to say that they are of the form $\nilk{N}{G}$ where $N$ is a uniform lattice of a connected and simply connected nilpotent Lie group $G$. It is known that any map $f$ on such a nilmanifold is homotopic to a map induced by an endomorphism $\varphi:G \to G,$ and D.~Anosov (\cite{a}, see also \cite{fh}) showed that $N(f)= |\det (I -\varphi_\ast)|$. Here $\varphi_\ast: \mathfrak{g}\to \mathfrak{g}$ is the differential of $\varphi$ at the identity element of $G$ (so the induced endomorphism on the Lie algebra $\mathfrak{g}$ of $G$).

\medskip

In the eighties, H.~Schirmer extended Nielsen theory, including Wecken's result, to the setting of $n$-valued maps $f: X \multimap X$ (\cite{s}, \cite{sch}). These are maps that attach to each $x\in X$ a subset $\{x_1,x_2, \ldots, x_n\}$  of $X$ with exactly $n$ elements. Also here we require some kind of continuity (see Section~\ref{sec: fixed point cl}). In \cite{b}, R.F.~Brown showed that any $n$-valued map on a circle is homotopic to an affine $n$-valued map (which is the natural analogue of the single-valued linear maps on nilmanifolds) and he was also able to compute the Nielsen number in this case. 

\medskip

In this paper we will show how to compute the Nielsen number of any affine $n$-valued map on any nilmanifold. In fact the main result of this paper is Theorem~\ref{hoofdformule} which nicely generalizes the formula of Anosov for the single-valued case. We do remark that in contrast to the single-valued case, it is no longer true that any $n$-valued map is homotopic to an affine $n$-valued map, but the class of affine $n$-valued maps is still very large and interesting. 

\medskip

This paper is organized as follows. In Section~\ref{sec: fixed point cl} we describe the approach to fixed point classes of $n$-valued maps as developed in \cite{bdds}. In the next section we recall the basic facts about nilpotent Lie groups and nilmanifolds. Section~\ref{sec: R for groups} is purely group theoretical in nature as we construct a way to count the number of generalized twisted conjugacy classes in the case of finitely generated torsion free nilpotent groups. This is in fact an algebraic way of counting fixed point classes (so for computing $R(f)$) for $n$-valued maps. In Section~\ref{bereken R} we introduce the class of affine $n$-valued maps and then apply the results from the previous section to find an easy formula for computing the Reidemeister number of such an affine $n$-valued map. In the last section finally, we compute the indices of all possible fixed point classes and this leads to the main result of this paper, namely a formula to compute the Nielsen number $N(f)$ of an affine $n$-valued map. 
\section{Fixed point classes of $n$-valued maps} \label{sec: fixed point cl}

Further on, we will make use of the theory developed in \cite{bdds}.
In this section, we recall some results from that paper.
\medskip \\
Let the space $X$ be a connected finite polyhedron.
The configuration space of $n$ ordered points on $X$ is
$$
F_n(X) = \{ (x_1,x_2,\dots,x_n ) \in X^n \; | \; x_i \neq x_j \text{ if } i \neq j \}
$$
and is topologized as a subspace of the product space $X^n.$
The configuration space of $n$ unordered points on $X$ is
$$
D_n(X) = \{ \{ x_1,x_2,\dots,x_n \} \subseteq X \; | \; x_i \neq x_j \text{ if } i \neq j \}
$$
and can be seen as the orbit space of $F_n(X)$ under the free action of the symmetric group $\Sigma_n.$
The quotient map $q: F_n(X) \to D_n(X) $ induces the quotient topology on $D_n(X).$
\medskip \\
For a natural number $n>0,$ an $n$-valued map $f: X \multimap X$ is a set-valued function which is continuous, that is, both lower and upper semi-continuous, and such that $f(x)$ contains exactly $n$ elements for each $x.$
We do not distinguish between the $n-$valued map $f:X\multimap X$ and the corresponding single-valued function $f:X \to D_n(X),$ which is continuous (as proved in \cite{bg}).
A fixed point of $f$ is an element $x$ such that $x \in f(x).$
The set of all fixed points of $f$ is denoted by $\Fix(f).$
\medskip \\
Let $p: \tilde X \to X$ be the universal cover of $X$ with covering group $\pi.$ The orbit configuration space is defined as
\begin{align*}
F_n(\tilde X, \pi)
&= \{ ( \tilde{x}_1, \tilde{x}_2, \ldots, \tilde{x}_n) \in \tilde{X}^n \;|\; p(\tilde{x}_i) \neq p(\tilde{x}_j) \mbox{ if } i \neq j \}
\\ &= \{ ( \tilde{x}_1, \tilde{x}_2, \ldots, \tilde{x}_n) \in \tilde{X}^n  \;|\; \pi \cdot \tilde{x}_i \neq \pi \cdot \tilde{x}_j \mbox{ if } i \neq j \}.
\end{align*}
Xicot\'encatl proved \cite{x} that
\[
p^n: F_n(\tilde X, \pi)  \to D_n(X):( \tilde{x}_1,\tilde{x}_2, \ldots, \tilde{x}_n) \mapsto \{  p(\tilde{x}_1), p(\tilde{x}_2), \ldots, p(\tilde{x}_n) \} 
\]
is a covering map with covering group $\pi^n \rtimes \Sigma_n$ and where the product in this group is given by
\[ (\alpha_1,\alpha_2, \ldots, \alpha_n;\sigma) (\beta_1,\beta_2, \ldots, \beta_n;\rho)=
( \alpha_1 \beta_{\sigma^{-1}(1)}, \alpha_2 \beta_{\sigma^{-1}(2)}, \ldots,\alpha_n \beta_{\sigma^{-1}(n)}; \sigma \circ \rho) .\]
Now we can lift any $n$-valued map $f: X \to D_n(X)$ to a map
$\bar f: \tilde{X} \to F_n(\tilde X, \pi).$
Since $F_n(\tilde X, \pi) \subset \tilde{X}^n, $ the lifting $\bar f$ splits in $n$ single-valued maps: $\bar{f}=(\bar{f}_1,\bar{f}_2,\ldots,\bar{f}_n).$
The maps $\bar{f}_i: \tilde{X} \to \tilde{X}$ are called lift-factors.
\[
\xymatrix@C=2.5cm{ \tilde{X} \ar[d]_p \ar[r]^-{\bar{f}=(\bar{f}_1,\bar{f}_2,\ldots,\bar{f}_n) }& F_n( \tilde{X}, \pi) \ar[d]^{p^n}\\
X\ar[r]_-f & D_n(X) }
\]
Analogously to the single-valued case, where there is an equivalence relation on the set of all liftings,
we define an equivalence relation on the set of all lift-factors of the $n$-valued map $f:$
$$ \bar{f}_i \sim_{lf} \bar{f}_j' \Leftrightarrow \exists \gamma \in \pi :  \bar{f}_i =  \gamma  \bar{f}_j' \gamma^{-1} .$$
The equivalence classes are called lift-factor classes, and the Reidemeister number $R(f)$ of $f$ is defined as the number of lift-factor classes.
If $n$-valued maps are homotopic, then they have the same Reidemeister number.
\medskip \\
Analogously to the single-valued case, where we have a similar property with liftings instead of lift-factors,
the following holds in the $n$-valued case:
\begin{theorem}\label{lcfpc}
Let $f \colon X  \to  
D_n(X)$ be an $n$-valued map
and $\bar f_i, \bar f'_j$
two lift-factors of $f.$
\begin{enumerate}[(a)]
\item If $\bar f_i \sim_{lf} \bar f'_j,$
then $p\Fix(\bar f_i) = p\Fix(\bar f'_j)$.
\item If  
$p\Fix(\bar f_i) \cap 
p\Fix(\bar f'_j) \ne \emptyset$ then $\bar f_i \sim_{lf} \bar f'_j.$
\end{enumerate}
\end{theorem} Hence, we can write the fixed point set
$
\Fix(f) = \bigcup_{\bar{f}_i} p \Fix( \bar{f}_i)
$ as the disjoint union
\[ \Fix(f) = \bigcup_{[\bar{f}_i]} p \Fix( \bar{f}_i), \]
where the union is taken over all lift-factor classes.
The sets $p \Fix( \bar{f}_i)$ are called fixed point classes of $f.$ It follows that the Reidemeister number, which is defined as the number of lift-factor classes, also equals the number of fixed point classes. Fixed point classes of $n$-valued maps were already defined by Schirmer in
1984, without the use of covering spaces (\cite{s}).
The preceding definition of fixed point classes agrees with the one from Schirmer in the case we are considering non empty fixed point classes.
Also the notions of index and Nielsen number were generalized by Schirmer to the $n-$valued setting. The Nielsen number $N(f)$ of $f$ is defined as the number of fixed point classes of non-zero index, so $N(f) \leq R(f).$
\medskip \\
We fix a lifting $\bar{f}^*=(\bar{f}^*_1,\bar{f}^*_2,\ldots,\bar{f}^*_n): \tilde{X} \to F_n( \tilde{X}, \pi) $ of $f,$ which we call the basic lifting. It holds that every lift-factor of $f$ can be written as $\alpha \bar{f}^*_i$ for an $\alpha \in \pi$ and $i \in \{ 1,2, \dots, n\}.$
The basic lifting of $f$ defines a function
$\psi_f:\pi \to \pi^n \rtimes \Sigma_n$
by the requirement that
\[ \forall \gamma\in \pi: \; \bar{f}^\ast \circ \gamma = \psi_f(\gamma) \circ \bar{f}^\ast .\]
We denote $\psi_f(\gamma)=(\phi_1(\gamma),\phi_2(\gamma), \ldots, \phi_n(\gamma); \sigma_\gamma)$ and it turns out that $\psi_f$ is a homomorphism.
The map $\sigma: \pi \to \Sigma_n: \gamma \mapsto \sigma_\gamma$ is also a homomorphism, but the maps $\phi_i: \pi \to \pi$ are not homomorphisms in general. Using the product structure of $\pi^n \rtimes \Sigma_n$, one finds that the maps $\phi_i$ satisfy the relation
\begin{equation}\label{f* na gamma}
    \bar{f}_i^\ast \circ \gamma = \phi_i(\gamma) \circ \bar{f}_{\sigma_\gamma^{-1}(i)}^\ast.
\end{equation}

We can translate
the equivalence relation $\sim_{lf}$ on the set of lift-factors to an equivalence relation on $\pi\times \{1,2,\ldots,n\}:$
\begin{align*}
\alpha \bar f^*_i \sim \beta \bar f^*_j
& \iff \exists \gamma \in \pi:
\left\{
\begin{array}{ll}
\sigma_\gamma (j) = i
\\
\alpha
= \gamma \beta \phi_j(\gamma^{-1}).
\end{array}
\right.
\end{align*}
Because the map $\sigma: \pi \to \Sigma_n: \gamma \mapsto \sigma_\gamma$ is a homomorphism, the relation
$$
i \sim j \iff \exists \gamma \in \pi: \sigma_\gamma(j) = i
$$
on $\{1,2,\dots,n\}$ is an equivalence relation. It divides the lift-factors of the basic lifting into equivalence classes, where $\bar{f}^*_i$ and $\bar{f}^*_j$ are in the same class if $i \sim j.$ Those classes are called the $\sigma-$classes of the basic lifting.
Let
$\bar{f}^*_{i_1},\bar{f}^*_{i_2},\dots,\bar{f}^*_{i_r}$
be representatives for the $\sigma-$classes.
It holds that
\begin{equation}\label{Som Reidemeister Getallen}
R(f) = R_{i_1}(f) + R_{i_2}(f) + \cdots + R_{i_r}(f)
\end{equation}
where $R_{i_k}(f)$ is the number of equivalence classes of the equivalence relation on $\pi$ given by
$$
\alpha\sim_{i_k} \beta \Leftrightarrow \exists \gamma \in S_{i_k}: \alpha = \gamma \beta \phi_{i_k} (\gamma^{-1}),
$$
with
$S_{i_k}=\{ \gamma\in \pi\;|\; \sigma_\gamma(i_k) =i_k \}.$
Even though the map $\phi_{i_k}: \pi \to \pi$
is not a homomorphism in general,
it is a homomorphism if we restrict the domain to $S_{i_k}$ (which is a finite index subgroup of $\pi$).
Denoting the equivalence class of $\gamma \in \pi$ by the relation $\sim_{i_k}$ with $[\gamma]_{i_k},$ we have
$$
\Fix(f) = \bigcup_{k=1}^r \bigcup_{[\gamma]_{i_k}} p \Fix( \gamma\bar{f}^*_{i_k}),
$$
where the last union is taken over all classes $[\gamma]_{i_k}$ with $\gamma \in \pi.$
The union is a disjoint union and some sets may be empty.

\section{Nilpotent groups, Malcev completions and nilmanifolds}\label{sectie3}
In this section we will recall some basic facts about finitely generated torsion free nilpotent groups and their Malcev completions.
For more details we refer the reader to \cite{c}, \cite{de}, \cite{m} and \cite{se}.

\medskip

For a group $N$, we will use $\gamma_i(N)$ ($i=1,2,\ldots$) to denote the terms of the lower central series of $N$. These terms are defined recursively via $\gamma_1(N)=N$ and $\gamma_{i+1}(N)=[N,\gamma_{i}(N)]$. Recall that a group is said to be nilpotent if there is a positive integer $c$ such that $\gamma_{c+1}(N)=1$. The smallest possible $c$ for which $\gamma_{c+1}(N)=1$ is called the nilpotency class of $N.$

\medskip

In this paper, we will be working with finitely generated torsion free nilpotent groups. Given such a group $N$, there exists a unique connected and simply connected nilpotent Lie group $G$ containing $N$ as a uniform lattice (i.e.\ as a discrete and cocompact subgroup). This Lie group $G$ is called the Malcev completion of $N$. Conversely, if $N$ is a uniform lattice of a connected and simply connected nilpotent Lie group $G$, then $N$ is finitely generated and torsion free. The spaces we are considering in this paper are the nilmanifolds. These are exactly the quotient spaces
$\nilk{N}{G}$ where $N$ is a uniform lattice of a connected and simply connected nilpotent Lie group $G$. Hence there is a one-to-one correspondence between nilmanifolds and finitely generated torsion free nilpotent groups.

\medskip

Now, let $N_1$ and $N_2$ be two such finitely generated torsion free nilpotent groups with Malcev completions $G_1$ and $G_2$ respectively. Then any morphism $\varphi:N_1 \to N_2$ extends uniquely to a Lie group morphism 
$\tilde{\varphi}: G_1 \to G_2$. Moreover, when $\varphi$ is an isomorphism (or injective or surjective), then so is $\tilde{\varphi}$.

\medskip 

Let $\lie$ denote the Lie algebra of a connected and simply connected nilpotent Lie group $G$. It is well known that the exponential map $\exp:\lie \to G$ is bijective in this case and we will denote the inverse by $\log:G \to \lie$. 
When $\varphi:G\to G$ is an endomorphism of the nilpotent Lie group $G$, we will use $\varphi_\ast:\lie \to \lie $
to denote the differential of $\varphi$ and we have a commutative diagram:
\[ 
\xymatrix{ G \ar[r]^\varphi\ar@<2pt>[d]^\log & G\ar@<2pt>[d]^\log\\
\lie\ar[r]_{\varphi_\ast}\ar@<2pt>[u]^\exp & \lie\ar@<2pt>[u]^\exp }
\]
In this paper we will need a more concrete description of the Malcev completion associated to a given finitely generated nilpotent group $N$. For our purposes it turns out that it is useful to work with the isolators of the terms of the lower central series of $N$. 
\begin{definition}
Let $H$ be a subgroup of a group $K$. By the isolator of $H$ in $K$, we will mean the following subset of $K$:
\[ \sqrt[K]{H}=\{ x \in K \mid \exists m >0:\; x^k \in H\}.\]
\end{definition}
When $N$ is a finitely generated torsion free nilpotent group of class $c$, we will use $N_i=\sqrt[N]{\gamma_i(N)}$ to denote the isolator of the $i$-th term of the lower central series of $N$. The $N_i$ are fully characteristic subgroups of $N$ and they form a central series 
\[1 = N_{c+1} \leq N_c \leq N_{c-1} \leq \cdots \leq N_2 \leq N_1 = N\]
in which each quotient $N_i/N_{i+1}$ is torsion free, in fact free abelian of finite rank, say $k_i$.
Note that $N_i$ is the smallest subgroup of $N$, containing $\gamma_i(N)$ and such that $N/N_i$ is torsion free.
In fact $N_i$ is the inverse image of the torsion group of $N/\gamma_i(N)$ under the natural projection $N\to N/\gamma_i(N)$. So $[N_i:\gamma_i(N)]<\infty$. When $G$ is the Malcev completion of $N$, then $N_i=\gamma_i(G) \cap N$
(because of Lemma 1.2.6 in \cite{d}).

\medskip

Now, choose elements $a_{i,j} \in N_i$ ($1 \leq i \leq c, 1 \leq j \leq k_i $) in such a way that the
$
\overline{ a }_{i,1} = a_{i,1} N_{i+1},\overline{ a }_{i,2} = a_{i,2} N_{i+1}, \dots, \overline{ a }_{i,k_i} = a_{i,k_i} N_{i+1}
$
form a free generating set of the free abelian group $N_i/N_{i+1} \cong \Z^{k_i}.$

\medskip

The elements $a_{1,1}, a_{1,2}, \dots, a_{c,k_c}$ form a so called Malcev basis of $N$ and we will say that this Malcev basis is compatible with the isolators of the lower central series. Any element $x\in N$ can now be written uniquely in the form
\begin{align}\label{x long notation}
    x = a_{1,1}^{z_{1,1}} a_{1,2}^{z_{1,2}} \cdots a_{1,k_1}^{z_{1,k_1}} a_{2,1}^{z_{2,1}} \cdots a_{c,k_c}^{z_{c,k_c}}
\end{align}
for some integers $z_{i,j} \in \Z.$
To shorten notation we will be using tuples
$
{z}_1 = (z_{1,1},z_{1,2}, \dots, z_{1,k_1}),
{z}_2 = (z_{2,1},z_{2,2}, \dots, z_{2,k_2}),
\dots,
{z}_c = (z_{c,1},z_{c,2}, \dots, z_{c,k_c})
$
and write
\begin{align}
    x = {a}_1^{{z}_1} {a}_2^{{z}_2} \cdots {a}_c^{{z}_c}
\end{align}
instead of (\ref{x long notation}) above.

\medskip

The following properties are well known.

\begin{theorem}
Let $N$ be a finitely generated torsion free nilpotent group and let $a_{1,1}, a_{1,2}, \dots, a_{c,k_c}$ be a Malcev basis of $N$ which is compatible with the isolators of the lower central series. Then
\begin{enumerate}
    \item $N_i = \text{grp}\{ a_{i,1}, a_{i,2},\dots, a_{c,k_c}\}$.
    \item There exist polynomial functions
    $p_i: \Q^{k_1+k_2+\cdots + k_{i-1}} \times \Q^{k_1+k_2+\cdots + k_{i-1}} \to \Q^{k_i}$ $(2\leq i \leq c)$ such that
    $\forall {x}_1,{y}_1 \in \Z^{k_1},$
    $\forall {x}_2,{y}_2 \in \Z^{k_2},$
    \dots,
    $\forall {x}_c,{y}_c \in \Z^{k_c}:$
    \begin{align*}
        &{a}_1^{{x}_1} {a}_2^{{x}_2} \cdots {a}_c^{{x}_c}
        {a}_1^{{y}_1} {a}_2^{{y}_2} \cdots {a}_c^{{y}_c}
        \\
        &=
        {a}_1^{{x}_1 + {y}_1}
        {a}_2^{{x}_2 + {y}_2 + p_2({x}_1 , {y}_1)}
        \cdots
        {a}_i^{{x}_i + {y}_i + p_{i}({x}_1 ,\dots,{x}_{i-1} , {y}_1, \dots, {y}_{i-1} )}
        \cdots
        {a}_c^{{x}_c + {y}_c + p_{c}({x}_1 , \dots,{x}_{c-1} ,{y}_1, \dots, {y}_{c-1} )}.
    \end{align*}
    \item There exist polynomial functions
    $q_i: \Q \times \Q^{k_1+k_2+\cdots + k_{i-1}} \to \Q^{k_i}$ $(2\leq i \leq c)$ such that
    $\forall {z}_1 \in \Z^{k_1},$
    $\forall {z}_2 \in \Z^{k_2},$
    \dots,
    $\forall {z}_c \in \Z^{k_c},$
    $\forall m \in \Z:$
    \begin{align*}
        ({a}_1^{{z}_1} {a}_2^{{z}_2} \cdots {a}_c^{{z}_c})^m
        =
        {a}_1^{ m {z}_1} {a}_2^{m {z}_2 + q_2(m,{z}_1)} \cdots {a}_c^{ m {z}_c + q_c(m,{z}_1,\dots,{z}_{c-1})}.
        \end{align*}
\end{enumerate}

\end{theorem}

Using the polynomials $p_i,$ we can define a group structure on $\R^{k_1 + k_2 + \cdots + k_c}$ by defining the product as
\begin{eqnarray*}
    \lefteqn{({x}_1,{x}_2,\dots,{x}_c)({y}_1,{y}_2,\dots,{y}_c)} \\
   & = &
    \big({x}_1+{y}_1, {x}_2+{y}_2+p_2({x}_1,{y}_1),\dots,{x}_c+{y}_c+p_c({x}_1,\dots, {x}_{c-1},{y}_1,\dots, {y}_{c-1})\big). 
\end{eqnarray*}
In this way we obtain a Lie group $G$ which is connected and simply connected and nilpotent.
The group $N$ is a uniform lattice of the group $G$ by identifying
${a}_1^{{z}_1} {a}_2^{{z}_2} \cdots {a}_c^{{z}_c}\in N$ with $({z}_1,{z}_2,\dots,{z}_c)\in G.$
So $G$ is the Malcev completion of $N.$ By abuse of notation we will use $ {a}_1^{{x}_1} {a}_2^{{x}_2} \cdots {a}_c^{{x}_c}$ to denote the element $({x}_1, {x}_2, \dots, {x}_c )$ of $G$.

Let $\mathfrak{g}$ be the Lie algebra of $G.$ If we let $A_{i,j} = \log(a_{i,j})$ ($1 \leq i \leq c, 1 \leq j \leq k_i$) then the set $ A_{1,1},A_{1,2}, \dots, A_{1,k_1},A_{2,1}, \dots, A_{c,k_c}$ is a basis of $\mathfrak{g}$ (as a vector space). Analogously as above, we will use ${x}_i {A}_i$ as a shorthand for 
$x_{i,1} A_{i,1} + x_{i,2} A_{i,2} + \cdots + x_{i,k_i} A_{i,k_i}$.

The lower central series of $\mathfrak{g}$ is defined analogously as for groups, i.e. $\gamma_1(\mathfrak{g}) = \mathfrak{g}$ and $\gamma_{i+1} (\mathfrak{g}) = [\mathfrak{g}, \gamma_i(\mathfrak{g})].$ It holds that $\exp(\gamma_i(\mathfrak{g})) = \gamma_i(G).$
A morphism $\varphi: G \to G$ induces a morphism on $\gamma_i(G)/\gamma_{i+1}(G) \cong \R^{k_i},$ where the isomorphism $\gamma_i(G)/\gamma_{i+1}(G) \to \R^{k_i}$ is given by $ a_{i,j} \gamma_{i+1}(G) \mapsto (0,\dots, 0, 1,0,\dots,0)^T $ with the ``1'' on the $j^{th}$ place. With $M_i$ we denote the matrix in $\R^{ k_i \times k_i}$ of that induced morphism.
Analogously, the differential $\varphi_\ast:\lie \to \lie $ of $\varphi$ induces a morphism on 
$\gamma_i(\lie)/\gamma_{i+1}(\lie) \cong \R^{k_i},$ where the isomorphism $\gamma_i(\lie)/\gamma_{i+1}(\lie) \to \R^{k_i}$ is given by $ A_{i,j} + \gamma_{i+1}(\lie) \mapsto (0,\dots, 0, 1,0,\dots,0)^T $ with the ``1'' on the $j^{th}$ place. It turns out that the matrix of that induced morphism is also equal to $M_i.$
\begin{theorem}\label{veel veeltermen}
Using the notations introduced above, we have the following properties.
\begin{enumerate}
    \item $\gamma_i(\mathfrak{g}) = \text{vct} \{ A_{i,1}, A_{i,2},\dots, A_{c,k_c}\}$.
    \item There exist polynomial functions
    $r_i: \R^{k_1+k_2+\cdots + k_{i-1}} \to \R^{k_i}$ $(2\leq i \leq c)$ such that
    $\forall {x}_1\in \R^{k_1},$
    $\forall {x}_2 \in \R^{k_2},$
    \dots,
    $\forall {x}_c \in \R^{k_c}:$
    \begin{align*}
        \exp( {x}_1 {A}_1 + {x}_2 {A}_2 + \cdots + {x}_c {A}_c) = {a}_1^{{x}_1} {a}_2^{{x}_2 + r_2({x}_1)} \cdots {a}_c^{{x}_c + r_c( {x}_1,{x}_2, \dots, {x}_{c-1} )}.
    \end{align*}
    \item There exist polynomial functions
    $s_i: \R^{k_1+k_2+\cdots + k_{i-1}}  \to \R^{k_i}$ $(2\leq i \leq c)$ such that
    $\forall {x}_1\in \R^{k_1},$
    $\forall {x}_2 \in \R^{k_2},$
    \dots,
    $\forall {x}_c \in \R^{k_c}:$
    \begin{align*}
        \log( {a}_1^{{x}_1} {a}_2^{{x}_2 } \cdots {a}_c^{{x}_c } ) = {x}_1 {A}_1 + \big({x}_2 + s_2({x}_1) \big){A}_2 + \dots + \big({x}_c + s_c({x}_1, {x}_2, \dots, {x}_{c-1} ) \big){A}_c.
    \end{align*} 
    \item If $\varphi: G \to G$ is a morphism, then
    there exist polynomial functions
    $P_i: \R^{k_1+k_2+\cdots + k_{i-1}}  \to \R^{k_i}$ and linear functions $P'_i: \R^{k_1+k_2+\cdots + k_{i-1}}  \to \R^{k_i}$ $(2\leq i \leq c)$ such that
    $\forall {x}_1\in \R^{k_1},$
    $\forall {x}_2 \in \R^{k_2},$
    \dots,
    $\forall {x}_c \in \R^{k_c}:$ 
    \begin{align*}
        \varphi (  {a}_1^{{x}_1} {a}_2^{{x}_2} \cdots {a}_c^{{x}_c}  ) 
        =
        {a}_1^{ M_1 {x}_1}
        {a}_2^{ M_2 {x}_2 + P_2({x}_1)}
         \cdots 
        {a}_c^{ M_c {x}_c + P_c({x}_1,{x}_2, \dots , {x}_{c-1})}
    \end{align*}
    and
    \begin{align*}
        \varphi_\ast (  {x}_1 {A}_1 & + {x}_2 {A}_2 + \cdots + {x}_c {A}_c  ) 
        \\
        &=
        M_1 {x}_1 {A}_1 + \big( M_2 {x}_2 +P'_2({x}_1) \big) {A}_2 + \cdots + \big( M_c {x}_c + P'_c( {x}_1,{x}_2,\dots, {x}_{c-1} ) \big) {A}_c.
    \end{align*}   
    
\end{enumerate}
\end{theorem}

\section{Generalized twisted conjugacy for nilpotent groups}
\label{sec: R for groups}

For single-valued self-maps $f$, the Reidemeister number $R(f)$ can be computed in an algebraic way by counting the number of twisted conjugacy classes, denoted by $R(f_\ast)$,  for the morphism $f_\ast$ induced on the fundamental group of the space. In the setting of $n$-valued maps, we need a generalization of this concept. 

Let $N$ be a group (which in our case will always be a finitely generated, torsion free nilpotent group). 
Suppose that $H \leq_f N$ is a finite index subgroup of $N$ and that we have a morphism $\phi: H \to N$. Then $\phi$ determines an equivalence relation $\sim_\phi$ on $N$ by setting
\[ \forall \alpha, \beta \in N: \; \alpha \sim_\phi \beta \iff \exists \gamma \in H : \alpha = \gamma \beta \phi(\gamma^{-1}).  \]

Note that when $H=G$, this is the usual twisted conjugacy relation (see e.g. \cite{dp}, \cite{fgw}, \cite{gw}, \cite{ms} and \cite{sen}), so we will refer to this relation as a generalized twisted conjugacy relation. 

We will use $R(\phi)$ to denote the number of equivalence classes of the generalized twisted conjugacy relation $\sim_\phi$ and refer to this number as the Reidemeister number of $\phi$, just like for the ordinary twisted conjugacy relation. This Reidemeister number $R(\phi)$ is either a positive integer or $\infty$.

Note that we already encountered this concept in Section~\ref{sec: fixed point cl} where we saw that 
$R(f_{i_k})$ was equal to the number of generalized twisted conjugacy classes of the morphism 
$\phi_{i_k}:S_{i_k} \to \pi$. So we can rewrite formula (\ref{Som Reidemeister Getallen}) as 
\[ R(f) = R_{i_1}(f) + R_{i_2}(f) + \cdots + R_{i_r}(f)=R(\phi_{i_1}) + R(\phi_{i_2})+ \cdots + R(\phi_{i_r}).\]

As we will need to compute the number of fixed point classes of an $n-$valued map on a nilmanifold, it will be crucial for us to understand how to compute the number of generalized twisted conjugacy classes in case $N$ is a finitely generated torsion free nilpotent group. This section is a rather technical section of which the main result in Theorem~\ref{R(f)} is a formula to compute this number.  But before we start developing this formula, let us first of all present a small example to illustrate this concept.

\begin{example}\label{simple example}
Let $N=\Z$, the additive group of integers and let $H = n \Z$ where $n=dm$ and $d$ and $m$ are positive integers. Take $\phi: H \to N : z \mapsto \displaystyle \frac{z}{m}$. 
Then we have that 
\begin{eqnarray*}
\forall z_1,z_2 \in N=\Z:\; z_1 \sim_\phi z_2 & \iff & \exists nz \in H= n \Z: \; z_1= n z + z_2 - \frac{n z}{m}\\
                                                                      & \iff & \exists z \in \Z:\; z_1 - z_2= \left( n - \frac{n}{m}  \right) z = (n-d) z. 
\end{eqnarray*}
In case $m\neq 1$ (so $d\neq n$), we see that $z_1\sim_\phi z_2 \iff z_1 \equiv z_2 \bmod  n-d $ and so $R(\phi)= n-d$.\\
On the other hand, when $m=1$ (so $d=n$) we get that  $z_1\sim_\phi z_2 \iff z_1 =z_2$ and so $R(\phi)=\infty$.
\end{example}

As indicated above, the aim of this section is to develop a formula for computing $R(\phi)$ for a morphism $\phi: H \to N$, where $H$ is a finite index subgroup of a torsion free finitely generated $c$-step nilpotent group $N.$ We again use the notation $N_i=\sqrt[N]{\gamma_i(N)}$ for the isolators of the terms of the lower central series of $N$ and analogously, we let $H_i =\sqrt[H]{\gamma_i(H)}$. It should be no surprise that the groups $H_i$ are in fact closely related to the $N_i$ as we show in the following lemma. We will use $G$ to denote the Malcev completion of $N.$ 
\begin{lemma}
With the notations introduced above, we have that $H_i=H \cap N_i$.\\
It follows that $H_i \leq_f N_i$ and that there is a natural isomorphism 
\[ \frac{H_i}{H_{i+1}} \cong \frac{H_i N_{i+1}}{N_{i+1}}\cdot\] 
\end{lemma}
\begin{proof}
Note that $G$ is also the Malcev completion of $H$ (since $H$ is of finite index in $N$), and so we have that $H_i= H \cap \gamma_i(G)= H \cap N \cap \gamma_i(G) = H \cap N_i$. 
The fact that  $H \leq_f N$, implies that $H\cap N_i \leq_f N_i$, hence $H_i \leq_f N_i$. We have the following equalities and the natural isomorphism from the second isomorphism theorem for groups: 
 \begin{align*}
        \frac{H_i}{H_{i+1}}
        =
        \frac{H_i}{H \cap N_{i+1}}
        =
        \frac{H_i}{H_i \cap N_{i+1}}
        \cong
        \frac{H_i N_{i+1}}{N_{i+1}} \cdot
    \end{align*}
\end{proof}

Using the natural isomorphism $\frac{H_i}{H_{i+1}} \cong \frac{H_i N_{i+1}}{N_{i+1}}$ from the lemma above, we may identify $\frac{H_i}{H_{i+1}}$ with a finite index subgroup 
of $\frac{N_i}{N_{i+1}}$ (because $H_i$,  and so also $H_i N_{i+1}$, is a finite index subgroup of $N_i$). Recall that in the previous section we already had chosen a fixed isomorphism 
$\frac{N_i}{N_{i+1}}\cong \Z^{k_i}$. Via this isomorphism the subgroup $ \frac{H_i N_{i+1}}{N_{i+1}}$ (and so also $\frac{H_i}{H_{i+1}}$) corresponds to a subgroup of finite index in $\Z^{k_i}$ which can be written as $B_i \Z^{k_i}$ for some matrix $B_i \in \Z^{k_i \times k_i}$. It holds that the index
\[ \left[\frac{N_i}{N_{i+1}} : \frac{H_i N_{i+1}}{N_{i+1}} \right] = [ \Z^{k_i}: B_i \Z^{k_i}] = | \det (B_i) |. \]

\begin{lemma}\label{induced morphism}
The morphism $\phi$ induces for each $i\in \{1,2,\dots, c\}$ a morphism 
\[
\phi_i: \frac{H_iN_{i+1}}{N_{i+1}}\to \frac{N_i}{N_{i+1}}: h N_{i+1} \mapsto \phi(h) N_{i+1}. 
\]
\end{lemma}
\begin{proof}
The morphism $\phi: H \to N$ has a unique extension (see Section \ref{sectie3}) to a morphism $\tilde\phi: G \to G$. As $\tilde \phi (\gamma_i(G)) \subseteq 
\gamma_i(G)$ we have that 
$\phi(H_i) = \tilde \phi ( H \cap \gamma_i(G)) \subseteq \tilde \phi (H) \cap \gamma_i(G) \subseteq N \cap \gamma_i(G)=N_i $ and
hence $\phi(h)\in N_i$ for all $h\in H_i$.  Moreover, if $h N_{i+1} = h' N_{i+1}$, then $h^{-1}h' \in N_{i+1}$ and 
$\phi (h^{-1}h') \in \tilde{\phi}(N_{i+1}) \subseteq N_{i+1}$ showing that $\phi(h) N_{i+1} = \phi (h')N_{i+1}$ and hence $\phi_i$ is a well defined map, which is now easily seen to be a morphism.
\end{proof}

Recall that we had identified $N_i/N_{i+1}$ with $\Z^{k_i}$ and $H_iN_{i+1}/N_{i+1}$ with the subgroup $B_i \Z^{k_i}$ of $N_i/N_{i+1}=\Z^{k_i}$ . Hence we will view 
$\phi_i$ as a morphism $\phi_i: B_i \Z^{k_i} \to \Z^{k_i}$.

\medskip

We will prove in this section that the product formula $R(\phi) = R(\phi_1) R(\phi_2) \cdots R(\phi_c)$ holds and therefore we first focus on the computation of $R(\psi)$ in case 
$\psi$ is a morphism $\psi: B \Z^k \to \Z^k$ for some matrix $B\in \Z^{k\times k}$, with $\det B \neq 0$ (meaning that $[\Z^k : B \Z^k]=|\det(B)|$). Note that in this situation, there exists a matrix $M\in \Q^{k\times k}$ such that $\psi: B \Z^k \to \Z^k: \vec{z} \mapsto M \vec{z}$. The following proposition shows how to express $R(\psi)$ in terms of $M$.
We will use $I_k$ to denote the $k\times k$ identity matrix and we also introduce the following notation:
\begin{definition} Let $a \in \R$, then we define $|a|_\infty$ as follows:
\[ |a|_\infty = \left\{ \begin{array}{cl}
|a| & \mbox{when } a\neq 0\\
\infty & \mbox{when } a=0.
\end{array}
\right.\]
\end{definition}

\begin{proposition}\label{formule abels}
Let $k$ be a positive integer and $B \in \Z^{k\times k}$ a matrix with $\det(B)\neq 0$. Let $\psi: B \Z^k \to \Z^k: \vec{z} \mapsto M \vec{z}$ be a morphism which is determined by a 
matrix $M \in \Q^{k\times k}$.  Then 
\[ R(\psi) =
 [\Z^k : B \Z^k ] | \det(I_k - M) |_\infty.
 \]
\end{proposition}

\begin{remark}$ $
\begin{itemize}
\item Although the matrix $M\in \Q^{k\times k}$, we must have that $MB$ is a matrix with integral entries because the image of $\psi$ is $MB \Z^k \subseteq \Z^k$.

\item 
Example~\ref{simple example} fits into the context of the above proposition, where $k=1$, $\psi=\phi$, $B$ is a $1\times 1$ matrix with entry $n$ and
$M$ is a $1 \times 1$ matrix with entry $\displaystyle \frac1m$. The proposition predicts that 
$R(\psi) = [\Z: n \Z] | \det( 1-\displaystyle \frac1m) | = n ( 1 - \displaystyle \frac1m) = n - d $ when $\det(1 - \displaystyle \frac1m) \neq 0$ ($\iff m \neq 1$) and 
$R(\psi) =\infty$ otherwise. This agrees with what we obtained by explicit calculation.
\end{itemize}
\end{remark}

\begin{proof}[Proof of Proposition \ref{formule abels}]
We have that for $\vec{z}_1, \vec{z}_2 \in \Z^k$:
\begin{eqnarray*}
\vec{z}_1 \sim_\psi \vec{z}_2 & \iff & \exists \vec{z} \in B\Z^k:\; \vec{z}_1  = \vec{z} + \vec{z}_2 - \psi( \vec{z}) \\
                                               & \iff & \exists \vec{z} \in B\Z^k:\; \vec{z}_1 -\vec{z}_2 = (I_k - M) \vec{z}.
\end{eqnarray*}
So, we see that $R(\psi)$ is the number of cosets of the subgroup $(I_k - M) B \Z^k$ of $\Z^k$ and hence we are looking for the index $[\Z^k: (I_k - M) B \Z^k]$. It is well known that for an integral $k\times k$ matrix $A$ it holds that $[\Z^k : A\Z^k]= |\det(A)|_\infty$. As $ (I_k - M) B  = B - MB$ is an integral matrix and $\det(B) \neq 0$, we get that 
\[ R(\psi) =|\det ((I_k-M) B)|_\infty = |\det(B)| |\det (I_k-M)|_\infty = [\Z^k : B \Z^k ] |\det (I_k-M)|_\infty.\]
\end{proof}

\begin{remark} \label{def Mi}
For each $i \in \{ 1,2, \dots, c\}$ we view 
$\phi_i$ as a morphism $B_i \Z^{k_i} \to \Z^{k_i}.$
We introduce the notation $M_i$ for the matrix in $\Q^{k_i \times k_i}$ associated with this morphism:
$$\phi_i : B_i \Z^{k_i} \to \Z^{k_i}: \vec{z} \mapsto M_i \vec{z} .$$
From Proposition \ref{formule abels}, it follows that
\[ R(\phi_i) =
 [\Z^{k_i} : B_i \Z^{k_i} ] | \det(I_{k_i} - M_i) |_\infty 
 = \bigg[\frac{N_i}{N_{i+1}} : \frac{H_i N_{i+1}}{N_{i+1}} \bigg] | \det(I_{k_i} - M_i) |_\infty  .
 \]
\end{remark}

The following lemma and its proof is a straightforward generalization of a result of Gon\c{c}alves and Wong in \cite{gw}.
\begin{lemma} \label{connection Reidemeister numbers} Consider groups $A_1 \leq_f A_2$ and $B_1 \leq_f B_2$ with $A_1 \trianglelefteq B_1,$ $A_2 \trianglelefteq B_2$ and $A_1= B_1\cap A_2.$
Assume we have the following commutative diagram with short exact rows
(with natural inclusion and projection morphisms).
\[ \begin{tikzcd}
    1 \arrow{r}  & A_1 \arrow{d}{\varphi'}\arrow{r}& B_1 \arrow{d}{\varphi}\arrow{r} & \frac{B_1}{A_1} \cong  \frac{B_1 A_2}{A_2}  \; \;  \arrow{r}\arrow{d}{\bar\varphi}  & 1\\
     1 \arrow{r}& A_2\arrow{r} & B_2 \arrow{r} & \frac{B_2}{A_2} \arrow{r} & 1
    \end{tikzcd}
\]
Then we have:
 \begin{enumerate}
        \item If $R(\bar\varphi)=\infty$, then $ R(\varphi) = \infty.$
        \item If $R(\bar\varphi)< \infty,\; |\Fix ( \bar \varphi)| <\infty$ and $R(\varphi')=\infty$, then $ R(\varphi)=\infty.$
        \item 
        If $A_2 \leq Z(B_2)$, then $ R(\varphi)\leq R(\bar\varphi)R(\varphi').$
    \end{enumerate}
\end{lemma}

\begin{remark}
In item 3.\ above, the product $ R(\bar\varphi)R(\varphi')$ is defined to be $\infty$ if $ R(\bar\varphi)=\infty$ or $R(\varphi')=\infty$. In the analogous result of \cite{gw} for usual twisted conjugacy, there is an equal sign ($=$) for item 3., but this should be an inequality ($\leq$).
\end{remark}

As $\phi_i$ can be seen as a morphism $\phi:B_i \Z^{k_i} \to \Z^{k_i}: \vec{z} \mapsto M_i \vec{z}  $ we can define the eigenvalues of such a morphism $\phi_i$ to be the eigenvalues of the rational matrix $M_i$. Note that this is independent of the choice of the isomorphism between $N_i/N_{i+1}$ and $\Z^{k_i}$.



It is obvious that 
$|\Fix(\phi_i)| = \infty $ if and only if  $M_i$ (and so by definition also $\phi_i$) has eigenvalue~1.

For nilpotent groups, we now have

\begin{lemma} \label{eigv 1}
Let $N$ be a finitely generated, torsion free nilpotent group of nilpotency class $c,$ $H \leq_f N$ and $\phi: H \to N$ a morphism. If $|\Fix(\phi)| = \infty,$ then there exists an $i \in \{ 1,2, \dots, c\}$ such that $\phi_i : \frac{H_iN_{i+1}}{N_{i+1}} \to \frac{N_{i}}{N_{i+1}} $ has eigenvalue 1.
\end{lemma}
\begin{proof} Let $h \in \Fix(\phi) \setminus \{ 1_H \}$ and let $i \in \{ 1,2, \dots, c\}$ be the biggest $i$ such that $h \in H_i.$ Then $h N_{i+1} \neq 1 N_{i+1}$ because otherwise $h \in N_{i+1} \cap H = H_{i+1}.$ So $h N_{i+1}$ is not the identity element, and $\phi_i( h N_{i+1} ) = h N_{i+1}.$ It follows that $\phi_i$ has eigenvalue 1.
\end{proof}

Let $X \leq N.$ We will use $\overline{X}$ to denote the group $\overline{X} := \frac{X N_c}{N_c}.$ So $\overline{N} = \frac{N N_c}{N_c} = \frac{N}{N_c}$ is a finitely generated, torsion free $(c-1)-$step nilpotent group, $\overline{(N_i)} = \frac{N_i N_c}{N_c} = \frac{N_i}{N_c}$ and $\overline{H} = \frac{H N_c}{N_c}.$ Note that we introduced brackets in the notation $\overline{(N_i)},$ in order to clearly show the difference with $(\overline{N})_i.$ The last group is $(\overline{N})_i = \big( \frac{N}{N_c} \big)_i = \sqrt[\frac{N}{N_c}\;\;]{ \gamma_i \big( \frac{N}{N_c} \big) }.$ However, there is not really any confusion possible, since we have the following lemma.

\begin{lemma}\label{overline N H}
Let $N$ be a finitely generated, torsion free nilpotent group  of nilpotency class $c$ and $H \leq_f N.$ Then $ \overline{H} \leq_f \overline{N}$ and we have that $\overline{(N_i)} = (\overline{N})_i$ and $\overline{(H_i)} = (\overline{H})_i$ for all $i \in \{ 1,2,\dots, c-1\}.$
\end{lemma}
\begin{proof}
To prove $ \overline{(N_i)} \supseteq(\overline{N})_i ,$ we choose an arbitrary $x \in \big( \frac{N}{N_c} \big)_i = \sqrt[\frac{N}{N_c}\;\;]{ \gamma_i \big( \frac{N}{N_c} \big) } .$ Then $ x = y N_c $ with $y \in N,$ and $ ( y N_c )^k = y^k N_c \in \gamma_i \big( \frac{N}{N_c} \big) $ for a positive integer $k.$ Since $ \gamma_i \big( \frac{N}{N_c} \big) = \frac{\gamma_i(N) N_c}{N_c},$ it follows that $y^k \in \gamma_i(N) N_c,$ so $y^k = \alpha \beta$ for certain $\alpha \in \gamma_i(N)$ and $\beta \in N_c = \sqrt[N]{\gamma_c(N)}.$ Let $l $ be a positive integer for which $\beta^l \in \gamma_c(N).$ Since $N_c \subseteq Z(N),$ we know that $[\alpha,\beta]=1,$ so
\begin{equation*}
        y^{kl} = (y^k)^l=(\alpha\beta)^l=\alpha^l\beta^l \in \gamma_i(N) \gamma_c(N) = \gamma_i(N).
\end{equation*}
It follows that $ y \in \sqrt[N]{\gamma_i(N)} = N_i$ and we conclude that $ x = y N_c \in \frac{N_i}{N_c} = \overline{(N_i)}.$
To prove $ \overline{(N_i)} \subseteq (\overline{N})_i,$ we take an arbitrary $z N_c \in \overline{(N_i)} = \frac{N_i}{N_c}.$ Since $z \in N_i = \sqrt[N]{\gamma_i(N)},$ there exists a positive integer $m $ such that $z^m \in \gamma_i(N).$ Then $(z N_c)^m = z^m N_c \in \frac{\gamma_i(N) N_c}{N_c} = \gamma_i\big( \frac{N}{N_c} \big) ,$ so $z N_c \in \sqrt[ \frac{N}{N_c} \; \;]{ \gamma_i\big(\frac{N}{N_c}\big) } = \big( \frac{N}{N_c} \big)_i = (\overline{N})_i. $ Hence we proved that $ \overline{(N_i)}=(\overline{N})_i.$ That gives us
\begin{align*}
        (\overline{H})_i = \overline{H} \cap (\overline{N})_i
        = \overline{H} \cap \overline{(N_i)}
        = \frac{H N_c}{N_c} \cap \frac{N_i N_c}{N_c}
        {\textbf{=}} \frac{ ( H \cap N_i ) N_c }{ N_c }
        = \frac{H_i N_c}{ N_c }
        = \overline{(H_i)}.
\end{align*}
\end{proof}

Recall that for a morphism $\phi: H \to N$ we defined the induced morphisms $\phi_i:\frac{H_iN_{i+1}}{N_{i+1}}\to \frac{N_i}{N_{i+1}}$ (see Lemma \ref{induced morphism}). Analogously, there is also an induced morphism $\overline{\phi}: \overline{H} = \frac{H N_c}{N_c} \to \overline{N} = \frac{ N}{N_c}: h N_c \mapsto \phi(h) N_c$ and hence also induced morphisms $(\overline{\phi})_i : \frac{\overline{H}_i \overline{N}_{i+1}}{\overline{N}_{i+1}} \to \frac{\overline{N}_i }{\overline{N}_{i+1}}: \overline{h} \; \overline{N}_{i+1} \mapsto \bar\phi(\bar h)\overline{N}_{i+1}.$

\begin{lemma} \label{R phi bar}
With the notation introduced above, we have that $R(\bar{(\phi)_i})=R(\phi_i)$ for all $i \in \{ 1,2,\dots, c-1\}.$
\end{lemma}

\begin{proof}
    We know that
    \begin{equation*}
        \frac{ \overline{H_{i}} \; \overline{N_{i+1}} }{ \overline{N_{i+1}} }
        = 
        \frac{ \frac{H_{i}N_c}{N_c} \frac{N_{i+1}}{N_c} }{ \frac{N_{i+1}}{N_c} }
        =
        \frac{ \frac{H_{i}N_{i+1}N_c}{N_c}  }{ \frac{N_{i+1}}{N_c} }
        \cong 
        \frac{ {H_{i}N_{i+1}N_c}  }{ {N_{i+1}}}
        =
        \frac{ {H_{i}N_{i+1}}  }{ {N_{i+1}}},
    \end{equation*}
    where the isomorphism holds because of the third isomorphism theorem. 
    Since
    $ \frac{ \overline{{N}_i} }{ \overline{{N}_{i+1}}}
    = \frac{ \frac{N_i}{N_c} }{ \frac{ N_{i+1} }{N_c}  } \cong \frac{ {N_i} }{ {N_{i+1}} }, $ we get the following commutative diagram:

    \centering
    \begin{tikzpicture}
\node[] (1) at (0,0) {$\frac{ {H_{i}N_{i+1}}  }{ {N_{i+1}}}$};
\node[] (2) at (3,0) {$\frac{ (\overline{H})_{i} \; (\overline{N})_{i+1} }{ (\overline{N})_{i+1} }$};
\node[] (3) at (6,0) {$\frac{ (\overline{N})_i }{ (\overline{N})_{i+1} }$};
\node[] (4) at (9,0) {$\frac{N_i}{N_{i+1}}$};

\draw [->] (1) -- (2);
\node[] (a) at (1.3,0.3) {$\cong$};
\draw [->] (2) -- (3);
\node[] (b) at (4.5,0.3) {$(\bar \phi)_i$};
\draw [->] (3) -- (4);
\node[] (c) at (7.5,0.3) {$\cong$};
\draw [->] (1) to [out=330,in=210] (4);
\node[] (e) at (4.5,-1.2) {$\phi_i$};
\end{tikzpicture}
 
    \noindent so we conclude that $R(\bar{(\phi)_i})=R(\phi_i).$
\end{proof}

\begin{proposition}\label{4.3.3 MM}
    With the previous notation, we have that $R(\phi) = \infty$ if and only if there exists an $i \in \{ 1,2,\dots, c \}$ such that $ \phi_i $ has eigenvalue 1. 
\end{proposition}

\begin{proof} Since for each $i \in \{ 1,2,\dots, c \}$ we have
    \begin{align*}
        \phi_i \text{ has eigenvalue 1 }
    &\iff \exists  \bar \gamma \in B_i \Z^{k_i} \setminus{ \{0\}}:
    M_i \bar \gamma = \bar \gamma
    \\
    &\iff \exists  \bar \gamma \in B_i \Z^{k_i} \setminus{ \{0\}}:
    (I_{k_i} - M_i) \bar\gamma = 0
    \\
    &\iff \det (I_{k_i} - M_i) = 0
    \\
    &\iff 
    |\det (I_{k_i} - M_i)|_\infty =\infty
    \\
    &\iff R(\phi_i)=\infty,
    \end{align*} 
    where the last equivalence holds because of Remark \ref{def Mi},
    we have to prove that
    \begin{align}\label{TB}
        R(\phi)=\infty \iff \exists i \in \{1,2,\dots,c\}: R(\phi_i)=\infty.
    \end{align}
    We prove (\ref{TB}) by induction on the nilpotency class $c$ of the group $N.$
    The case $c=1$ is trivial.
    Suppose (\ref{TB}) holds for all finitely generated, torsion free nilpotent groups $N$ of nilpotency class $\leq c-1,$
    for all $H\leq_{f}N$ and morphisms $\phi:H\to N.$
    Let now $N$ be a finitely generated, torsion free nilpotent group of class $c,$ $H\leq_f N$ and $\phi: H \to N$ a morphism.
    Consider the following commutative diagram.
    
   \begin{center}
    \begin{tikzpicture}[x=0.8cm,y=0.6cm]
\node[] (1) at (0,0) {$1$};
\node[] (2) at (3,0) {$H_c$};
\node[] (3) at (6,0) {$H$};
\node[] (4) at (9,0) {$\frac{H}{H_c} \cong \frac{H N_c}{N_c} $};
\node[] (5) at (12,0) {$1$};

\draw [->] (1) -- (2);
\draw [->] (2) -- (3);
\draw [->] (3) -- (4);
\draw [->] (4) -- (5);

\node[] (6) at (0,-3) {$1$};
\node[] (7) at (3,-3) {$N_c$};
\node[] (8) at (6,-3) {$N$};
\node[] (9) at (9.4,-3) {$\frac{N}{N_c}  $};
\node[] (10) at (12,-3) {$1$};

\draw [->] (6) -- (7);
\draw [->] (7) -- (8);
\draw [->] (8) -- (9);
\draw [->] (9) -- (10);


\draw[->] (2) -- (7);
\node[] (d) at (3.8,-1.5) {$\phi'=\phi_c$};
\draw[->] (3) -- (8);
\node[] (d) at (6.25,-1.5) {$\phi$};
\draw[->] (9.4,-0.45) -- (9);
\node[] (d) at (9.65,-1.5) {$\bar \phi$};
\end{tikzpicture}
\end{center}

    The group $\frac{N}{N_c} = \overline{N}$ is finitely generated, torsion free and nilpotent of class $c-1$ and has 
    the central series

    \[
        1 \leq \frac{N_{c-1}}{N_c}
        \leq 
        \dots 
        \leq
        \overline{N}_{i} = \frac{N_i}{N_c}
        \leq
        \dots 
        \leq
         \frac{N_2}{N_c}
         \leq \frac{N}{N_c} = \overline{N}.
   \]
    
    \noindent Suppose $R(\phi)=\infty.$
    If $R(\bar\phi) = \infty,$
    then the induction hypothesis implies that
    there exists an $i \in \{ 1,2,\dots, c-1\}$ such that $R(\bar{(\phi)_i})=\infty$ and so
     $R(\phi_i)=\infty,$ by Lemma \ref{R phi bar}.
    If $R(\bar\phi) < \infty,$ then $R(\phi_c)$ must be infinite, 
    since if $R(\phi_c)<\infty,$
    then $\infty = R(\phi) \leq R(\bar \phi) \cdot R(\phi_c) < \infty$
    because of Lemma \ref{connection Reidemeister numbers}(3).
    \\ \\
    Now suppose there exists an $i \in \{1,2,\dots,c\}$ with $R(\phi_i)=\infty.$\\
    If $i\in \{1,2,\dots,c-1\},$
    then it follows from Lemma \ref{R phi bar}
    that $R(\bar{(\phi)_i})=\infty.$
    The induction hypothesis implies that $R(\bar \phi)=\infty,$
    and so $R(\phi)=\infty$ by Lemma \ref{connection Reidemeister numbers}(1).
    \\
    If $i=c$ and $R(\phi_k)<\infty$ for all $k\in\{1,2,\dots,c-1\}$ 
    then, by applying Lemma \ref{eigv 1}
    on the group ${ \overline{N}=}\frac{N}{N_c},$
    it follows that $|\Fix(\bar \phi)| < \infty.$
    By Lemma \ref{connection Reidemeister numbers}(2) we can finally conclude that $R(\phi)=\infty.$
\end{proof}
We can now prove the product formula:
\begin{theorem}\label{R(f) product} 
Let $N$ be a finitely generated, torsion free nilpotent group of nilpotency class $c,$ $H \leq_f N$ and $\phi: H \to N$ a morphism. Define $N_i = \sqrt[N]{ \gamma_i(N)}$ and $H_i = \sqrt[H]{ \gamma_i(H)}$ for each $i \in \{ 1,2,\dots,c\}$ and let $\phi_i: \frac{H_i N_{i+1}}{N_{i+1}} \to \frac{N_i}{N_{i+1}}$ be the induced morphism. Then
$$R(\phi)= \prod_{i=1}^c R(\phi_i).$$
\end{theorem}

\begin{proof}
    This follows from Proposition \ref{4.3.3 MM} if $R(\phi)$ or $R(\phi_i)$  is infinite for certain $i.$
    We may thus assume that $R(\phi)$ and all $R(\phi_i)$ are finite.
    \\
    The case $c=1$ is trivial.
    Let $c>1$ and consider $\bar \phi: {\frac{HN_c}{N_c} } \to \frac{N}{N_c}.$
    By Lemma \ref{R phi bar} it suffices to prove that $R(\phi) = R(\phi_c)R(\bar\phi).$
    \\
    If $R(\bar\phi)$ is infinite, then from Lemma \ref{connection Reidemeister numbers}(1) it follows that $R(\phi)=\infty,$ a contradiction, so $R(\bar\phi)$ must be finite.
    \\ Choose $y_1, y_2,\dots, y_{m_1} \in N$ such that 
    $\bar y_1 =y_1N_c,\bar y_2 =y_2N_c, \dots, \bar y_{m_1} =y_{m_1}N_c$
    are representatives for $\sim_{\bar \phi}.$
    Choose representatives $x_1, x_2,\dots, x_{m_2} \in N_c$ 
    for $\sim_{\phi_c}.$
    \\
    We claim that $x_iy_j$ ($i\in\{1,2,\dots,m_2\},j\in\{1,2,\dots,m_1\}$) are representatives for the Reidemeister classes of $\phi.$
    \\ \\
    First we show that all $x_iy_j$ represent different Reidemeister classes of $\phi.$
    Suppose $x_{i_1}y_{j_1} \sim_\phi x_{i_2}y_{j_2},$
    so there exists a $\gamma \in H$ such that
    \begin{align} \label{g}
        x_{i_1}y_{j_1} = \gamma x_{i_2}y_{j_2} \phi(\gamma^{-1}).
    \end{align}
    Projection on $\frac{N}{N_c}$ gives
    $\bar x_{i_1} \bar y_{j_1} = \bar\gamma \bar x_{i_2} \bar y_{j_2} \bar\phi(\bar\gamma^{-1}),$
    so $ \bar y_{j_1} \sim_{\bar\phi}  \bar y_{j_2}.$
    This means that 
    $y_{j_1}N_c \sim_{\bar\phi} y_{j_2}N_c$ and 
    so $j_1 = j_2.$

    If $\gamma\in H_c,$ we get $\phi(\gamma^{-1})\in N_c \leq Z(N)$
    so $ x_{i_1} = \gamma x_{i_2} \phi(\gamma^{-1}).$
    Then $x_{i_1} \sim_{\phi_c} x_{i_2},$ thus $i_1 = i_2.$
    \\
    If $\gamma \notin H_c,$
    let $k \in \{ 1,2,\dots,c-1 \}$ be the biggest $k$ such that $\gamma\in H_k$.
    Since $x_{i_1},x_{i_2} \in N_c \leq Z(N)$,
    from (\ref{g}) it follows that
    \begin{align*}
        x_{i_1}x^{-1}_{i_2} = y_{j_1}^{-1}\gamma y_{j_1} \phi(\gamma^{-1})
        =y_{j_1}^{-1}\gamma y_{j_1} \gamma^{-1}\gamma \phi(\gamma^{-1})
    \end{align*}
    and therefore $ x_{i_1}x_{i_2}^{-1} N_{k+1} =  y_{j_1}^{-1}\gamma y_{j_1} \gamma^{-1}\gamma \phi(\gamma^{-1}) N_{k+1} =  [y_{j_1},\gamma^{-1}] \gamma \phi(\gamma^{-1}) N_{k+1}. $
    Since $[y_{j_1},\gamma^{-1}] \in [N,N_k] \leq N_{k+1}$ and $x_{i_1}x_{i_2}^{-1} \in N_c \leq N_{k+1},$
    we obtain $ N_{k+1} = \gamma \phi(\gamma^{-1}) N_{k+1}$
    and consequently $\gamma N_{k+1}$ is a fixed point of $\phi_k.$
    By Proposition \ref{4.3.3 MM},
    since $R(\phi)<\infty,$
    $\phi_k$ has no non-trivial fixed points.
    It follows that $\gamma N_{k+1} = N_{k+1},$
    so $\gamma\in N_{k+1},$ a contradiction.
    \\
    {
    We find that all $x_i y_j$ represent different Reidemeister classes of $\phi,$ so $R(\phi) \geq m_1 \cdot m_2.$
    Because of Lemma \ref{connection Reidemeister numbers}(3), it holds that $R(\phi) \leq R(\bar \phi)R(\phi_c) = m_1 \cdot m_2.$
    We can conclude that $R(\phi) = m_1 \cdot m_2 = R(\bar \phi)R(\phi_c).$
    
    }
\end{proof}

Using this product formula, we can now express $R(\phi)$ in terms of the index $[N:H]$ and the matrices $M_i$ introduced above. 

\begin{theorem}\label{R(f)}
Let $N$ be a finitely generated, torsion free nilpotent group of nilpotency class $c,$ $H \leq_f N$ and $\phi: H \to N$ a morphism. Define $N_i = \sqrt[N]{ \gamma_i(N)}$ and $H_i = \sqrt[H]{ \gamma_i(H)}$ for each $i \in \{ 1,2,\dots,c\}$ and let $\phi_i: \frac{H_i N_{i+1}}{N_{i+1}} \to \frac{N_i}{N_{i+1}}$ be the induced morphisms with associated matrices $M_i$ as in Remark~\ref{def Mi}. Then
$$R(\phi) = [N:H] \cdot |\det  (I_{\sum k_i} - M)|_\infty \text{ with } 
M = \begin{pmatrix} M_1 & 0 & & 0
\\
0 & M_2 &\cdots &0 
\\ \vdots & \vdots& \ddots & \vdots
\\ 0  & 0 & \cdots  & M_c
\end{pmatrix}\cdot$$
\end{theorem}

\begin{proof}
    Since the matrix $I_{\sum k_i} - M$ is a block diagonal matrix,
    we have that $\det (I_{\sum k_i} - M) = \prod_{i=1}^c \det (I_{k_i} - M_i),$
    so
    $ |\det (I_{\sum k_i} - M) |_\infty = \prod_{i=1}^c |\det (I_{k_i} - M_i) |_\infty. $
    From Theorem \ref{R(f) product} and Remark \ref{def Mi}, it  follows that
    \begin{align*}\label{R 1e uitdr}
        R(\phi)= \prod_{i=1}^c \bigg[ \frac{N_i}{N_{i+1}} : {  \frac{H_i N_{i+1}}{N_{i+1}} } \bigg] \cdot |\det  (I_{k_i} - M_{i})|_\infty =\bigg( \prod_{i=1}^c \bigg[ \frac{N_i}{N_{i+1}} : {  \frac{H_i N_{i+1}}{N_{i+1}} }  \bigg] \bigg) \cdot  |\det (I_{\sum k_i} - M) |_\infty .
    \end{align*}
    We still have to prove
    \begin{equation}\label{nogTB}
        [N:H] = \prod_{i=1}^c \bigg[ \frac{N_i}{N_{i+1}} : { \frac{H_i N_{i+1}}{N_{i+1}} } \bigg].
    \end{equation}
    For a group $G$ with subgroup $S$ and normal subgroup $N,$ we know that
    \begin{align}\label{product indices}
        \big[G:S\big] = \bigg[ \frac{G}{N} : \frac{SN}{N} \bigg] \cdot [N : S \cap N].
    \end{align} 
    We use that to prove (\ref{nogTB}) by induction.
    Since $H \leq N$ and $N_2 \trianglelefteq N,$ we get
    $$ \big[N:H\big] 
    = \bigg[ \frac{N}{N_2} : {  \frac{HN_2}{N_2} } \bigg] \cdot [N_2 : H \cap N_2] = \bigg[ \frac{N}{N_2} : \frac{HN_2}{N_2} \bigg] \cdot [N_2 : H_2],
    $$
    so (\ref{nogTB}) is proven for $c=2.$
    \\ \\
    Now we assume
    \begin{align}\label{IH}
        &[N:H] \nonumber \\
        &= \bigg[ \frac{N}{N_2} : { \frac{H N_2}{N_2} } \bigg]
                \cdot
                \bigg[ \frac{N_2}{N_3} : { \frac{H_2 N_3}{N_3} } \bigg]
                \cdots
                \bigg[ \frac{N_{k-2}}{N_{k-1}} : { \frac{H_{k-2} N_{k-1} }{N_{k-1}} }\bigg]
                \cdot
                [N_{k-1} : H_{k-1}].
    \end{align}
    Because of (\ref{product indices}), we have
    \begin{align*}
     \big[N_{k-1}:H_{k-1}\big] 
    &= \bigg[ \frac{N_{k-1}}{N_k} : \frac{H_{k-1}N_k}{N_k} \bigg] \cdot [N_k : H_{k-1} \cap N_k]
   \\ &= 
    \bigg[ \frac{N_{k-1}}{N_k} : \frac{H_{k-1}N_k}{N_k} \bigg] \cdot [N_k : H_k].
    \end{align*}
    By substituting in (\ref{IH}), it follows that
    { 
    \begin{align*}
        [N:H]& \\
        = \bigg[ \frac{N}{N_2} &: {  \frac{H N_2}{N_2} } \bigg]
                \cdot
                \bigg[ \frac{N_2}{N_3} : {  \frac{H_2 N_3}{N_3} } \bigg]
                \cdots
                \\
                & \cdots \bigg[ \frac{N_{k-2}}{N_{k-1}} : {  \frac{H_{k-2} N_{k-1} }{N_{k-1}} }\bigg]
                \cdot
                \bigg[ \frac{N_{k-1}}{N_k} : \frac{H_{k-1}N_k}{N_k} \bigg] \cdot [N_k : H_k].
    \end{align*}
    }
    By continuing this until $k=c+1$, we finally obtain \eqref{nogTB}.
\end{proof}

\section{Reidemeister numbers of affine $n$-valued maps}\label{bereken R}

Now we turn to nilmanifolds $\nilk{N}{G}$ again.
Here $N$ is a uniform lattice of a connected and simply connected nilpotent Lie group $G.$ The natural projection $p: G \to \nilk{N}{G}$ is the universal covering of the nilmanifold $\nilk{N}{G}$ and $N,$ acting by left translations on $G,$ is the covering group of the projection.
\medskip \\
As before, an $n-$valued map $f: \nilk{N}{G} \multimap \nilk{N}{G} $ will be written as a single-valued map $f: \nilk{N}{G} \to D_n\left(\nilk{N}{G}\right).$ As explained in Section \ref{sec: fixed point cl}, the map $f$ lifts to a split map
$\bar{f}^*=(\bar{f}^*_1,\bar{f}^*_2,\ldots,\bar{f}^*_n): G\to F_n( G,N), $
i.e. we have a commutative diagram
\[
\xymatrix@C=2.5cm{ G \ar[d]_p \ar[r]^-{\bar{f}^*=(\bar{f}^*_1,\bar{f}^*_2,\ldots,\bar{f}^*_n) }& F_n( G, N) \ar[d]^{p^n}\\
\nilk{N}{G} \ar[r]_-f & D_n\left(\nilk{N}{G}\right). }
\]
Here each $\bar{f}^*_i$ is a self-map of $G.$
\medskip \\
Among the self-maps of $G,$ there are the so-called affine maps. An affine map of $G$ is a map of the form
$h: G \to G : x \mapsto g \varphi(x)$
where $\varphi: G \to G$ is a morphism of Lie groups ($\varphi \in\text{End}(G))$ and $g \in G.$
Note that in the case where $G=\R^n$ (the abelian Lie group), an affine map $h$ of $\R^n$ maps an element $x$ to $g+\varphi(x),$ where $\varphi \in \text{End}(\R^n)$ is a linear map and $g$ is the translational part of $h.$ So this is really a usual affine map of $\R^n.$

\begin{definition} Let $f: \nilk{N}{G} \to D_n\left(\nilk{N}{G}\right)$ be an $n-$valued map and $\bar{f}^*=(\bar{f}^*_1,\bar{f}^*_2,\ldots,\bar{f}^*_n): G \to F_n(G,N)$ be a lifting of $f.$ Then we say that $f$ is an affine $n-$valued map if each $\bar{f}^*_i$ is an affine map of $G.$
\end{definition}

\begin{remark}
The definition of $f$ being affine is independent of the chosen lifting $\bar{f}^*,$ since for another lifting $\bar{f}^{*'}=(\bar{f}^{*'}_1,\bar{f}^{*'}_2,\ldots,\bar{f}^{*'}_n)$ we will have that $\bar{f}^{*'}_i = \alpha \bar{f}^*_j$ for some $j \in \{1,2,\dots,n\}$ and $\alpha \in N.$ As $\bar{f}^*_j$ is assumed to be affine, then composition with a left translation $\alpha$ is still affine.
\end{remark}

\begin{remark}
Recall that we are interested in maps up to homotopy (since Reidemeister numbers and Nielsen numbers are homotopy invariant). From this point of view it makes sense to pay special attention to the class of $n-$valued maps that are affine. Indeed, it is well known that any map on an (infra-)nilmanifold is homotopic to an affine map (see \cite{l}) and it is also known that any $n-$valued map on a circle is homotopic to an affine map (\cite{b}). Although it is no longer true that all $n$-valued maps on a nilmanifold will be homotopic to an affine $n$-valued map, this still holds for a very large class of $n$-valued maps and moreover affine $n$-valued maps are a very natural class of maps to study. 
\end{remark}
From now on, we consider an affine $n-$valued map $f: \nilk{N}{G} \to D_n\left(\nilk{N}{G}\right)$ and we fix a lifting
$\bar{f}^*=(\bar{f}^*_1,\bar{f}^*_2,\ldots,\bar{f}^*_n): G \to F_n(G,N).$
So there are $g_i \in G$ and $\varphi_i \in \text{End}(G)$ ($i \in \{ 1,2,\dots,n \}$) such that $\bar{f}^*_i: G \to G: x \mapsto g_i\varphi_i(x).$ We will use the notation $\bar{f}^*_i = (g_i,\varphi_i)$ from now on.
\medskip \\
Recall that the chosen lifting $\bar{f}^*$ determines a morphism
$\psi_f:N \to N^n \rtimes \Sigma_n$
with
$\psi_f(\gamma)=(\phi_1(\gamma),\phi_2(\gamma), \ldots, \phi_n(\gamma); \sigma_\gamma)$ such that
$
    \bar{f}_i^\ast \circ \gamma = \phi_i(\gamma) \circ \bar{f}_{\sigma_\gamma^{-1}(i)}^\ast.
$
\medskip \\
We now fix an $i$ and consider the stabilizer
$$
H \; (= S_i) = \{ \gamma\in N | \sigma_\gamma(i) = i \} \leq_f N.
$$
We want to compute $R_i(f)$ and so we have to consider the morphism $\phi_i: H \to N.$
So from now onwards $\phi_i$ will play the role of the morphism $\phi$ from Section \ref{sec: R for groups}.
We still use $N_1,N_2,\dots,N_c$ to denote the isolators of the terms of the lower central series of $N$ and the induced morphisms on the factors are now denoted by
$$
(\phi_i)_l : \frac{H_l N_{l+1}}{N_{l+1}} \to \frac{N_l}{N_{l+1}}.
$$
We fix a Malcev basis $a_{1,1}, a_{1,2}, \dots,a_{c,k_c}$ which is compatible with the isolators of the lower central series and use
$
A_{p,q}=\log(a_{p,q})
$
as before.
The choice of this Malcev basis determines an isomorphism
$ \frac{N_l}{N_{l+1}} \cong \Z^{k_l}, $
under which
$\frac{H_l N_{l+1}}{N_{l+1}}$ corresponds with $B_l\Z^{k_l}$
and this allows us to identify $(\phi_i)_l$ with the morphism
$$
(\phi_i)_l: B_l \Z^{k_l} \to \Z^{k_l}: \vec{z} \mapsto M_l \vec{z}
$$
for some $M_l \in \Q^{k_l \times k_l}.$ The element of $M_l$ in the $p^{th}$ row and $q^{th}$ column is denoted $m^{(l)}_{p,q}.$
\medskip \\
As $H\leq_f N,$ we know that the Malcev completion $G$ of $N$ is also the Malcev completion of $H.$ Therefore the morphism $\phi_i: H \to N$ extends uniquely to a morphism $\phi_i: G \to G$ (which we abusively denote by the same symbol).
There is a close relationship between $\phi_i$ and $\varphi_i.$
\medskip
\begin{lemma}\label{phi_i conjugatie}
        Let $i\in \{1,2,\dots,n\}$ and $\bar f^*_i = (g_i,\varphi_i).$
        Then for all $\gamma \in H = \{ \gamma \in N | \sigma_\gamma(i) = i \},$
        it holds that
        $$ \phi_i (\gamma) = g_i \varphi_i(\gamma) g_i^{-1}. $$ 
    \end{lemma}
    \begin{proof} For all $\gamma \in N$ we have that 
    $\bar f^*_i \circ \gamma = \phi_i(\gamma)\circ\bar f^*_{\sigma_\gamma^{-1}(i)},$
    so for each $\gamma \in H $ this gives
    $\bar f^*_i \circ \gamma = \phi_i(\gamma)\circ\bar f^*_{i}.$
    This means that 
    $ g_i \varphi_i (\gamma x) = \phi_i(\gamma)g_i \varphi_i(x)$
    for all $x \in G$
    so, since $\varphi_i$ is a morphism, 
    $ g_i \varphi_i (\gamma) \varphi_i(x) = \phi_i(\gamma)g_i \varphi_i(x).$
    We get that
    $ g_i \varphi_i (\gamma) g_i^{-1} = \phi_i(\gamma)$
    for all $\gamma \in H,$
    because $g_i$ and $\varphi_i(x)$ are invertible.
    \end{proof}
    As we have the unique extension $\phi_i: G \to G$ we can consider its differential $\phi_{i*}: \lie \to \lie.$
    \begin{theorem}
    Using the notation above, we have that the matrix expression of $\phi_{i*}$ with respect to the basis $A_{1,1}, A_{1,2},\dots,A_{c,k_c}$ is given by a matrix of the form
    \begin{equation*}
    \begin{pmatrix}
    M_1         &0      & \hdots&       &  0
    \\ *        & M_2   &       &       &0
    \\ \vdots   & *     &\ddots &       &\vdots
    \\ \vdots   &       &       & \ddots&0
    \\ *        &*      &\hdots &   *   &M_c 
    \end{pmatrix}.
    \end{equation*}
    \end{theorem}
    \begin{proof}
    Since $H$ is a finite index subgroup of $N,$
    there exists a normal subgroup $H_0$ of finite index in $N$ with $H_0 \leq H.$
    If $m = [N:H_0],$ then $\gamma^m \in H_0 \leq H$ for all $\gamma \in N.$
    Let  $p \in \{1,2,\dots,c \}$ and $q \in \{1,2,\dots,k_p \},$
    then $\log(a^m_{p,q}) = m A_{p,q}.$
    \medskip \\
    Under the isomorphism $N_p/N_{p+1}\cong \Z^{k_p},$
    the element $\bar{a}_{p,q} = a_{p,q} N_{p+1}$ corresponds to the column vector $(0,0,\dots, 0, 1, 0, \dots,0)^T,$ where the 1 is on the $q^{th}$ place.
    It follows that
    $( \phi_i )_p (\bar{a}^m_{p,q}) = \bar{a}_{p,1}^{m\cdot m_{1,q}^{(p)}} \bar{a}_{p,2}^{m\cdot m_{2,q}^{(p)}} \cdots \bar{a}_{p,k_p}^{m\cdot m_{k_p,q}^{(p)}},$
    so there exists an element $n_{p+1} \in N_{p+1}$ such that
    $$ \phi_i  ({a}^m_{p,q}) = {a}_{p,1}^{m\cdot m_{1,q}^{(p)}} {a}_{p,2}^{m\cdot m_{2,q}^{(p)}} \cdots {a}_{p,k_p}^{m\cdot m_{k_p,q}^{(p)}} n_{p+1}.$$
    This implies that
    $
    \phi_{i*} ( m A_{p,q} ) = m ( m_{1,q}^{(p)} A_{p,1} + m_{2,q}^{(p)} A_{p,2} + \cdots + m_{k_p,q}^{(p)} A_{p,k_p}) + x
    $
    for a certain $x \in \gamma_{p+1}(\mathfrak{g}),$ and so
    $$
    \phi_{i*} ( A_{p,q} ) =  m_{1,q}^{(p)} A_{p,1} + m_{2,q}^{(p)} A_{p,2} + \cdots + m_{k_p,q}^{(p)} A_{p,k_p} + x'
    $$
    for an $x' \in \gamma_{p+1}(\mathfrak{g}).$
    
    \end{proof}

    \begin{theorem} \label{Formule Reidemeister}
Let $\nilk{N}{G}$ be a nilmanifold and 
$f: \nilk{N}{G} \to D_n\left(\nilk{N}{G}\right)$ an affine $n-$valued map.
Choose a basic lifting $\bar{f}^*=(\bar{f}^*_1 = (g_1,\varphi_1),\bar{f}^*_2 = (g_2,\varphi_2),  \ldots,\bar{f}^*_n = (g_n,\varphi_n)),$
with $ g_i \in G$ and $\varphi_i \in \text{End}(G)$ inducing a morphism
$\psi_f:N \to N^n \rtimes \Sigma_n$
with
$\psi_f(\gamma)=(\phi_1(\gamma),\phi_2(\gamma), \ldots, \phi_n(\gamma); \sigma_\gamma)$.
We have that for each $i\in \{1,2,\dots,n \}$
$$ R_i(f) = [N : S_i] \cdot|\det ( I - \varphi_{i*} ) |_\infty $$
where $S_i=\{ \gamma \in N\mid \sigma_\gamma(i) = i\}$
and
$$R(f) = \sum_{i=1}^n | \det( I - \varphi_{i*}  ) |_\infty .$$
In the formula above $\det( I - \varphi_{i*}  )$ is computed by using the matrix expression of $\varphi_{i\ast}$ with respect to a chosen basis of $\mathfrak{g}$.
\end{theorem}

\begin{proof} As the determinants $\det ( I - \varphi_{i*} )$ are independent of the chosen basis for the vector space $\mathfrak{g}$ we can fix a Malcev basis $a_{1,1}, a_{1,2}, \dots,a_{c,k_c}$ for $N$ which is compatible with the isolators of the lower central series and use the
$
A_{i,j}=\log(a_{i,j}) 
$ as a basis of $\mathfrak{g}$ as before.\\
    Let $i$ be in $\{1,2,\dots,n \}.$
    If we denote conjugation with an element $g\in G$ as
    $\mu(g): G \to G: x \mapsto \mu(g)(x) = g x g^{-1},  $
    then, by Lemma \ref{phi_i conjugatie}, the morphism $\phi_i: G \to G$ is equal to $\phi_i = \mu(g_i) \circ \varphi_i,$
    so $\varphi_i = \mu(g_i^{-1}) \circ \phi_i.$
    It follows that $\varphi_{i*} = \mu(g_i^{-1})_* \circ \phi_{i*}: \mathfrak{g} \to \mathfrak{g} $
    and therefore the matrix expression of $\varphi_{i*}$
    with respect to the basis $A_{1,1}, A_{1,2},\dots,A_{c,k_c}$
    is
    \begin{equation*}
    \begin{pmatrix}
    I_{k_1}     &0          & \hdots&       &  0
    \\ *        & I_{k_2}   &       &       &0
    \\ \vdots   & *         &\ddots &       &\vdots
    \\ \vdots   &           &       & \ddots&0
    \\ *        &*          &\hdots &   *   &I_{k_c} 
    \end{pmatrix}
    \cdot
    \begin{pmatrix}
    M_1         &0      & \hdots&       &  0
    \\ *        & M_2   &       &       &0
    \\ \vdots   & *     &\ddots &       &\vdots
    \\ \vdots   &       &       & \ddots&0
    \\ *        &*      &\hdots &   *   &M_c 
    \end{pmatrix}
    = 
    \begin{pmatrix}
    M_1         &0      & \hdots&       &  0
    \\ *'        & M_2   &       &       &0
    \\ \vdots   & *'     &\ddots &       &\vdots
    \\ \vdots   &       &       & \ddots&0
    \\ *'        &*'      &\hdots &   *'   &M_c 
    \end{pmatrix},
    \end{equation*}
    because conjugation by $g_i^{-1}$ induces the identity on each factor $\gamma_l(G)/\gamma_{l+1}(G)$.\\
    Together with Theorem \ref{R(f)}, this implies that
    $$R_i(f) = [N:S_i] \cdot \Big|\det  \Big(I - 
    \begin{pmatrix} M_1 & 0 & & 0
\\
0 & M_2 &\cdots &0 
\\ \vdots & \vdots& \ddots & \vdots
\\ 0  & 0 & \cdots  & M_c
\end{pmatrix}\Big)\Big|_\infty 
= [N:S_i] \cdot \Big|\det  \Big(I - \varphi_{i*}\Big)\Big|_\infty. $$
Let now
$\bar{f}^*_{i_1},\bar{f}^*_{i_2},\dots,\bar{f}^*_{i_r}$
be representatives for the $\sigma-$classes of $\bar{f}^*.$
As discussed in Section \ref{sec: fixed point cl}, we get
\begin{equation*}
R(f) = \sum_{k=1}^r R_{i_k}(f) 
= \sum_{k=1}^r [N : S_{i_k}] \cdot \Big|\det  \Big(I - \varphi_{i_k*}\Big)\Big|_\infty
=
\sum_{i=1}^n \Big|\det  \Big(I - \varphi_{i*}\Big)\Big|_\infty
,
\end{equation*}
where the last equality holds because the index $[N : S_{i_k}]$ equals the number of elements in the $\sigma-$class of $i_k.$
\end{proof}

\section{Nielsen numbers of affine $n$-valued maps}
To compute the Nielsen number of an $n$-valued affine map $f$, we will decompose $f$ into its irreducible components as described by Staecker in Section 4 of \cite{st}. Let $f:X \to D_n(X)$ be an $n$-valued map. An $m$-valued map $g:X \to D_m(X)$ (with $1\leq m \leq n$) is a submap of $f$ when $g(x) \subseteq f(x)$ for all $x\in X$. The map $f$ is said to be irreducible when $f$ does not have an $m$-valued submap with $m<n$. 

For any $n$-valued map $f:X \to D_n(X)$ there exist $r$ (unique up to order) irreducible $n_i$-valued maps $f_i:X\to D_{n_i}(X)$
with $n_1 + n_2 +\cdots + n_r=n$ for some $r\in \{1,2,\ldots,n\}$, such that 
$f(x) = f_1(x) \cup f_2(x) \cup \cdots \cup f_r(x)$ for all $x\in X$. We will refer to the $f_i$ as the irreducible components of $f$.

Let $\bar f^\ast= (\bar f^\ast_1, \bar f^\ast_2, \ldots, \bar f^\ast_n): \tilde X \to F_n(\tilde X, \pi) $ be a fixed basic lifting of $f$ determining a morphism $\psi_f: \pi \to \pi{^n} \rtimes \Sigma_n: \gamma  \mapsto \psi_f(\gamma) = (\phi_1(\gamma),\phi_2(\gamma), \ldots, \phi_n(\gamma); \sigma_\gamma)$ as before and recall that we divided the lift-factors $\bar f^\ast_i$ into equivalence classes, called $\sigma$--classes, by putting $\bar f^\ast_i \sim \bar f^\ast_j \Leftrightarrow i \sim j \Leftrightarrow \exists \gamma \in \pi: \sigma_\gamma(j) =i$. 
When the $\sigma$--class of $\bar f^\ast_i$ is given by $\{ \bar f^\ast_{j_1} = \bar f^\ast_i, \bar f^\ast_{j_2}, \ldots,\bar f^\ast_{j_{m}} \}$, then 
the map $g: X \to D_m(X): x  \mapsto \{ p (\bar f^\ast_{j_1} (\tilde x)), p (\bar f^\ast_{j_2} (\tilde x)), \ldots, p (\bar f^\ast_{j_m} (\tilde x))\}$, where 
$\tilde x \in \tilde{X}$ is an element with $p(\tilde x)=x$, is a well defined irreducible $m$-valued map which is one of the irreducible components of $f$ and all irreducible components of $f$ are obtained in this way. So the number of irreducible components of $f$ equals the number of $\sigma$--classes.

Note that, when $f_1, f_2, \ldots, f_r $ are the irreducible components of the $n$-valued map $f:X \to D_n(X)$, formula \eqref{Som Reidemeister Getallen} can now also be interpreted as 
\[ R(f) = R(f_1) + R(f_2) + \ldots + R(f_r).\]

We also have the following result.

\begin{proposition}\label{bereken Nielsen} (Corollary 4.7 from \cite{st}) Let $f:X\to D_n(X)$ be an $n$-valued map with irreducible components $f_1, f_2, \ldots, f_r$, then 
\[ N(f)= N(f_1) + N(f_2) + \cdots + N(f_r).\]
    
\end{proposition}

We will now apply this to affine $n$-valued maps on nilmanifolds. We start with the following observation.

\begin{lemma} Let $\nilk{N}{G}$ be a nilmanifold and let $f:\nilk{N}{G} \to D_n\left(\nilk{N}{G}\right)$ be an affine $n-$valued map with 
basic lifting $\bar f^\ast =(\bar f^\ast_1, \bar f^\ast_2, \ldots, \bar f^\ast_n)$. If $\bar{f}^*_i = (g_i, \varphi_i)$ and $\bar{f}^*_j= (g_j, \varphi_j)$ are in the same $\sigma-$class, then $\varphi_i = \varphi_j.$ 
\end{lemma}

\begin{proof}
    The lift-factors $\bar{f}^*_i$ and $\bar{f}^*_j$ are in the same $\sigma-$class
    if and only if
    there exists a $\gamma \in N$ such that $\sigma_\gamma^{-1} (i) = j.$
    For the element $\gamma$ it holds that
    $\bar{f}^*_i \circ \gamma = \phi_i(\gamma) \circ \bar{f}^*_j.$
    Let $\gamma' = \phi_i(\gamma).$
    Then for all $x \in G$ we have that
    $(\bar{f}^*_i \circ \gamma) (x) = (\gamma' \circ \bar{f}^*_j) (x) ,$
    so
    $g_i \varphi_i (\gamma x) = \gamma' g_j \varphi_j(x).$
    Taking $x=1$ we get $g_i \varphi_i (\gamma) = \gamma' g_j.$
    It follows that $g_i \varphi_i(\gamma x) = g_i \varphi_i(\gamma ) \varphi_i( x) = \gamma' g_j \varphi_i(x)$
    and so
    $ g_i \varphi_i(\gamma x) = \gamma' g_j \varphi_j(x)  $
    reduces to
    $ \gamma' g_j \varphi_i(x) = \gamma' g_j \varphi_j(x),$
    from which we get that
    $\varphi_i(x) = \varphi_j(x)$ for all $x\in G$ so $\varphi_i = \varphi_j.$
\end{proof}

\begin{remark}
 Note that when $f$ is an irreducible affine $n$-valued map, there is only one $\sigma$--class and hence all the linear parts $\varphi_i$ of all lift-factors are equal to each other and so we will just need one endomorphism $\varphi=\varphi_i\in {\text{End}}(G)$.  By Proposition~\ref{bereken Nielsen} it is enough to be able to compute the Nielsen number of irreducible affine $n$-valued maps. In order to achieve this, we will make a distinction in two cases, depending on whether or not the differential $\varphi_\ast$ of the linear part has an eigenvalue 1 or not. 
\end{remark}

Before we start considering the two cases, we prove one more result that will be useful for both.

\begin{lemma} \label{psi bijective}
    Let $G$ be a connected and simply connected nilpotent Lie group.
    Suppose that $\varphi \in \text{End}(G)$ is such that $\varphi_*\in \text{End}(\mathfrak{g})$ does not have eigenvalue 1.
    Then the map
    $\psi : G \to G: x \mapsto \varphi(x)x^{-1} $ is bijective.
\end{lemma}

\begin{proof}
    We prove this by induction on the nilpotency class $c$ of $G.$
    If $c=1,$ then $G \cong \R^k$ and $\varphi = \varphi_*.$
    In this case, the map $\psi $ is $\varphi - \text{Id}_{\R^k},$
    which is bijective because $\det(\varphi - \text{Id}_{\R^k}) \neq 0.$
    Suppose $c>1$ and $\psi$ is bijective if $G$ has nilpotency class $c-1.$
    Note that $\varphi( \gamma_c(G) ) \subseteq \gamma_c(G),$
    so $\psi$ induces a map
    $$
    \bar{\psi} : \frac{G}{\gamma_c(G)} \to \frac{G}{\gamma_c(G)}:
    x \gamma_c(G) \mapsto \varphi(x) x^{-1} \gamma_c(G). 
    $$
    In fact $ \bar{\psi} $ is the map corresponding to the endomorphism
    $$
    \bar{\varphi} : \frac{G}{\gamma_c(G)} \to \frac{G}{\gamma_c(G)} : x \gamma_c(G) \mapsto \varphi(x) \gamma_c(G)
    $$
    and so $\bar{\psi}$ is bijective because $\bar \varphi_*$ also does not have 1 as an eigenvalue.
    Let $\varphi' = \varphi_{\mkern 4mu \vrule height 2ex\mkern4mu \gamma_c(G)},$
    then $\varphi' ( = \varphi'_*)$ does not have 1 as an eigenvalue
    and also
    $$
    {\psi'} : {\gamma_c(G)} \to {\gamma_c(G)}:
    x  \mapsto \varphi'(x) x^{-1} 
    $$
    is bijective.
    To prove the surjectivity of $\psi,$
    suppose $y \in G.$
    Then there exists an $x' \gamma_c(G)$ such that
    $ \bar{\psi} ( x' \gamma_c(G) ) = y \gamma_c(G) .$
    This implies that 
    $ \varphi(x') x'^{-1} = y z $
    for some $z \in \gamma_c(G).$
    Since $\psi'$ is bijective,
    there exists a $z_1 \in \gamma_c(G)$
    with $\varphi(z_1) z_1^{-1} = z^{-1}.$
    Since $z_1, \varphi(z_1) \in \gamma_c(G) \subseteq Z(G),$ we have
    $$
    \varphi(x' z_1) (x' z_1)^{-1}
    =
    \varphi(x') x'^{-1} \varphi(z_1) z_1^{-1}
    =
    yzz^{-1}
    =
    y.
    $$ 
    Hence for $x = x'z_1$
    we have that
    $\psi(x) = y,$
    showing that $\psi$ is surjective.
    To prove the injectivity of $\psi,$
    suppose $\psi(x) = \psi(y).$
    Then $\varphi(x) x^{-1} = \varphi(y) y^{-1}$ so 
    $\varphi(x) x^{-1} \gamma_c(G) = \varphi(y) y^{-1} \gamma_c(G).$
    Since $\bar \psi$ is bijective,
    it follows that
    $ x \gamma_c(G) = y \gamma_c(G) $
    and so
    $x = yz$ for some $z \in \gamma_c(G).$
    But then
    $$
    \varphi(x) x^{-1}
    =
    \varphi(yz) (yz)^{-1}
    =
    \varphi(y) y^{-1} \varphi(z) z^{-1}
    =
    \varphi(x) x^{-1} \varphi(z) z^{-1},
    $$
    which implies that $\varphi(z) z^{-1} = 1.$
    It follows that $z=1$ because $\psi'$ is bijective
    and consequently $x=y,$
    which was to be shown.
\end{proof}

\subsection{The Nielsen number in case $\varphi_{\ast}$ has eigenvalue 1} 
In this subsection we will show that $f$ is homotopic to an $n$-valued map without fixed points, from which we will be able to conclude that $N(f)=0$.

 \begin{lemma}\label{epsilon} Let $l$ be a positive integer.
    For each $i \in \{ 1,2,\dots,l\},$
    let
    $h_i:\R^n \to \R$ 
    be a map.
    Then there exists an $\epsilon \in \R$
    such that
    $$ h_i(z_1, z_2, \dots, z_n) + \epsilon \neq 0$$
    for all
    $z_1, z_2, \dots, z_n \in \Z$
    and all $i \in \{ 1,2,\dots,l\}.$
    \end{lemma}
    
    \begin{proof}
        The set
        $$
        A = \bigcup_{i=1}^l \{ h_i(z_1, z_2, \dots, z_n) | z_j \in \Z\} 
        $$
        is countable, so there exists an $r \in \R$ such that $r \notin A.$
        Let $\epsilon = -r.$
        If there should be an $i\in \{1,2,\dots,l\}$ and $z_1, z_2, \dots, z_n \in \Z$ with 
        $h_i(z_1, z_2, \dots, z_n) + \epsilon = 0,$
        then $h_i(z_1, z_2, \dots, z_n) = r,$ so $r \in A,$ a contradiction. 
        Therefore we found an $\epsilon \in \R$ such that for all $i\in \{1,2,\dots,l\}$ and for all $z_1, z_2, \dots, z_n \in \Z:$
        $$
        h_i(z_1, z_2, \dots, z_n) + \epsilon \neq 0.
        $$
    \end{proof}

\begin{lemma}\label{alpha for eig 1} Let $G$ be a connected and simply connected $c$-step nilpotent Lie group.
    Let $\varphi \in \text{End}(G) $ with $\varphi_*$ having eigenvalue 1.
    Let $N$ be a lattice of $G$ and $g_1, g_2, \dots, g_n \in G.$
    Then there exists an $\alpha \in G$ such that
    $g_k \varphi(x) \alpha \neq \gamma x$
    for all $x \in G, \gamma \in N$ and $k=1,2,\dots,n.$
\end{lemma}

\begin{proof}
    Let $a_{1,1},a_{1,2},\dots, a_{c,k_c} $ be a Malcev basis as usual and use the matrices $M_i$ for representing the induced endomorphism
    on $\gamma_i(\mathfrak{g})/\gamma_{i+1}(\mathfrak{g})$ as before.
    Since $\varphi_*$ has eigenvalue 1, there is a matrix $M_i$ with eigenvalue 1.
    Let $r$ be the smallest index such that $M_r$ has eigenvalue 1.
    Without loss of generality, we may assume that $r=c,$
    because we can prove the lemma for $\frac{G}{\gamma_{r+1}(G)}$ and then it follows for $G.$
    So we assume $G$ is $c$-step nilpotent, $M_c$ has eigenvalue 1 and $M_i$ does not have eigenvalue 1 for $1\leq i \leq c-1$.
    Since $\gamma_c(G) \cong \R^{k_c},$ we can find a vector space basis
    $\{v_1, v_2, \dots, v_{k_c} \}$ of $\R^{k_c}$ such that the matrix expression of $\varphi_c$ is of the form
    \begin{equation*}
    \left(
    \begin{array}{c|ccc}
    1           &0          & \hdots    &  0 
    \\ \hline * & *         &   \hdots  &*
    \\ \vdots   &  \vdots   &\ddots     &\vdots
    \\ *        &*          &\hdots     &*
    \end{array}
    \right)
    \end{equation*}
    with respect to this basis.
    Let
    $V = \text{span}\{v_1\}$
    and
    $U = \text{span}\{v_2,v_3, \dots, v_{k_c} \},$
    then $U$ and $V$ are subgroups of $\gamma_c(G)$ and hence of $G$
    and so we will keep using the multiplicative notation.
    We get $ \gamma_c(G) = VU$ and $V = \{ v_1^r | r \in \R \}.$
    Note that by the choice of basis we have that
    $\varphi(v_1) = \varphi_c(v_1) = v_1 u$ for some $u$ in $U$
    and
    $\varphi(u) = \varphi_c(u) \subseteq U.$
    We claim that for a suitable choice of $\epsilon$
    we can take $\alpha = v_1^\epsilon.$
    Suppose that $ x = a_1^{x_1} a_2^{x_2} \cdots a_c^{x_c}, \gamma = a_1^{z_1}a_2^{z_2} \cdots a_c^{z_c} $
    and $k$ satisfy
    $g_k \varphi(x) v_1^\epsilon = \gamma x.$
    Then, by considering the equation modulo $\gamma_c(G),$
    we get that
    $\bar{g}_k \bar{\varphi}(\bar{x}) = \bar{\gamma} \bar{ x}.$
    It follows that
    \begin{equation}\label{bar phi}
        \bar{\varphi}(\bar{x})\bar{x}^{-1} = \bar{g}_k^{-1} \bar{\gamma}.
    \end{equation}
    From Lemma \ref{psi bijective},
    it follows that
    $\bar{x} = \bar{a}_1^{x_1} \bar{a}_2^{x_2} \cdots \bar{a}_{c-1}^{x_{c-1}} $
    is uniquely determined by $\bar{g}_k$  and $\bar{\gamma}.$
    So $x_1, x_2,\dots, x_{c-1}$ are uniquely determined by $k$ and $z_1, z_2, \dots, z_{c-1}.$
    This implies that $x' = a_1^{x_1} a_2^{x_2} \cdots a_{c-1}^{x_{c-1}}$
    is uniquely defined by $k$ and $z_1, z_2, \dots, z_{c-1}$
    and $x = x' a_c^{x_c}.$
    We can write
    $a_c^{x_c}$ as $a_c^{x_c} = v_1^{\alpha_1} u_1$
    with $\alpha_1 \in \R$ and $u_1 \in U.$
    Now (\ref{bar phi}) implies that
    $ \varphi(x') x'^{-1} = g_k^{-1} \gamma v_1^{\alpha_2} u_2 $
    with $\alpha_2 \in \R$ and $u_2 \in U.$
    Remark that also $\alpha_2$ is now completely determined by $k$ and $z_1, z_2, \dots, z_c$ because $x'$ is.
    So $\alpha_2 = h_k(z_1,z_2,\dots,z_c)$ for some map
    $h_k : \R^{k_1 + k_2 + \cdots + k_c} \to \R.$
    The equation
    $g_k \varphi(x) v_1^\epsilon = \gamma x$
    can now be written as
    \begin{equation} \label{Ak}
        g_k \varphi(x' v_1^{\alpha_1} u_1) v_1^\epsilon = \gamma x' v_1^{\alpha_1} u_1.
    \end{equation}
    Note that
    $ \varphi(v_1^{\alpha_1} u_1 ) = v_1^{\alpha_1} u_3  $
    for some $u_3 \in U$
    and so (\ref{Ak}) can be written as
    \begin{align*}
        &g_k \varphi(x') v_1^{\alpha_1} u_3 v_1^\epsilon = \gamma x' v_1^{\alpha_1} u_1
        \\
        \iff
        &\varphi(x') x'^{-1} = g_k^{-1} \gamma v_1^{-\epsilon} u_4 \text{ for some } u_4 \in U
        \\
        \iff
        &g_k^{-1} \gamma v_1^{\alpha_2} u_2 = g_k^{-1} \gamma v_1^{-\epsilon} u_4
        \\
        \iff
        &v_1^{\alpha_2} u_2 = v_1^{-\epsilon} u_4 \\
        \iff & v_1^{h_k(z_1,z_2,\ldots,z_c)} u_2 = v_1^{-\epsilon} u_4. \\
    \end{align*}
    We know by Lemma \ref{epsilon} that we can choose $\epsilon$
    in such a way that none of the equations
    $h_k(z_1,\dots,z_c) + \epsilon =0$
    is satisfied, which proves the lemma.
\end{proof}

\begin{theorem}\label{case 1}
    Let $f: \nilk{N}{G} \to D_n\left(\nilk{N}{G}\right)$
    be an irreducible affine $n-$valued map
    with basic lifting
    $\bar{f}^* = (g_1 \varphi, g_2 \varphi, \dots, g_n \varphi),$
    where $g_1, g_2, \dots, g_n \in G$ and $\varphi \in \text{End}(G).$
    If $\varphi_*$ has eigenvalue 1, then $N(f)=0.$
\end{theorem}

\begin{proof}
    Applying Lemma \ref{alpha for eig 1}
    gives an element $\alpha \in G$ such that
    for all $x\in G$ and $\gamma \in N$ it holds that
    \begin{align}\label{no fixed points}
        g_1 \varphi(x) \alpha \neq \gamma x,
        g_2 \varphi(x) \alpha \neq \gamma x,
        \dots,
        g_n \varphi(x) \alpha \neq \gamma x.        
    \end{align}
    Define
    $$
        \bar g^* : G \to G^n:
        x \mapsto \big(g_1 \varphi(x)\alpha, g_2 \varphi(x)\alpha, \dots, g_n \varphi(x)\alpha \big).
    $$
    As we know that $\bar f^\ast(x) = \big(g_1 \varphi(x), g_2 \varphi(x), \dots, g_n \varphi(x) \big)\in F_n(G,N)$, it easily follows that also $\bar g^\ast (x)\in F_n(G,N)$, so {we} will view $\bar g^\ast$ as a map 
    from $G$ to $F_n(G,N)$.
    
    Let $\gamma \in N$ and $x \in X.$ Using (\ref{f* na gamma}),
    we get 
    \begin{align*}
            \bar g^*( \gamma x)
            &= \big(\bar f^*_1 (\gamma x)\alpha, \bar f^*_2(\gamma x)\alpha, \dots, \bar f^*_n (\gamma x)\alpha \big)
            \\
            &= \big(\phi_1(\gamma) \bar f^*_{\sigma_\gamma^{-1}(1)}(x)\alpha,
            \phi_2(\gamma) \bar f^*_{\sigma_\gamma^{-1}(2)}(x)\alpha,
            \dots,
            \phi_n(\gamma) \bar f^*_{\sigma_\gamma^{-1}(n)}(x)\alpha \big)
            \\
            &= \psi_f(\gamma) \bar g^*(x).
        \end{align*}
    It follows that $\bar g^*$ is the basic lifting of a well defined $n$-valued map
    $g: \nilk{N}{G} \to D_n\left(\nilk{N}{G}\right).$
    Because of (\ref{no fixed points}), the map $g$ has no fixed points.
    The maps $\bar f^*$ and $\bar g^*$ are homotopic by
            $$ \bar F^*: G\times [0,1]\to F_n(G,N) :(x,t) \mapsto  \big(g_1 \varphi(x)\alpha^t, g_2 \varphi(x)\alpha^t, \dots, g_n \varphi(x)\alpha^t \big). $$
    This homotopy induces a homotopy between $f$ and $g$. As $g$ has no fixed points, $N(g)=0$ and hence also $N(f)=0.$
\end{proof}

\begin{remark}
For the $n$-valued affine map $f$ from the previous theorem we have that 
\[R(f)=  \sum_{i=1}^n \Big|\det  \Big(I - \varphi_{*}\Big)\Big|_\infty = n  \Big|\det  \Big(I - \varphi_{*}\Big)\Big|_\infty =\infty.\] 
So $f$ has infinitely many fixed point classes, all of index 0.
\end{remark}

\subsection{The Nielsen number in case $\varphi_{\ast}$ does not have eigenvalue 1}

In this case we want to show that each fixed point class exists of a single fixed point and that the index is non-zero.     
 
\begin{lemma}
Let $f: \nilk{N}{G} \to D_n\left(\nilk{N}{G}\right)$
    be an irreducible affine $n-$valued map
    with basic lifting
    $\bar{f}^* = (g_1 \varphi, g_2 \varphi, \dots, g_n \varphi),$
    where $g_1, g_2, \dots, g_n \in G$ and $\varphi \in \text{End}(G).$
    If $\varphi_*$ does not have eigenvalue 1, each fixed point class of $f$ consists of a single point.
\end{lemma}   

\begin{proof}
   Any lift{-}factor of $f$ is of the form $\gamma g_i \varphi$ for some $\gamma \in N$ and some $i=1,2,\ldots, n$. The corresponding fixed point class is then $p \Fix(\gamma g_i \varphi)$. It is enough to show that $\Fix(\gamma g_i \varphi)$  is a singleton. Note that $x\in \Fix(\gamma g_i \varphi)$ if and only if $\varphi(x)x^{-1}= g_i^{-1} \gamma^{-1}$. By Lemma~\ref{psi bijective}, there is a unique such $x\in G$.   
\end{proof}

We now prove that the index of such a fixed point class is never 0.

\begin{lemma}
Let $f: \nilk{N}{G} \to D_n\left(\nilk{N}{G}\right)$
    be an irreducible affine $n-$valued map
    with basic lifting
    $\bar{f}^* = (g_1 \varphi, g_2 \varphi, \dots, g_n \varphi),$
    where $g_1, g_2, \dots, g_n \in G$ and $\varphi \in \text{End}(G).$
    If $\varphi_*$ does not have eigenvalue 1, {then} each fixed point class of $f$ has index $\pm 1$.
\end{lemma} 
\begin{proof}
As in the previous lemma, a fixed point class is determined by a lift-factor, which is of the form $\gamma g_i \varphi$. The index of the corresponding fixed point class of $f$ is the same as the index of the unique fixed point of the map $\gamma g_i \varphi: G \to G,$ as defined by Schirmer on page 210 of \cite{s}.
As in Section~\ref{sectie3}, using Malcev coordinates, we identify $G$ with $\R^K$ (where $K=k_1+k_2+\cdots+k_c$) via 
$G \to \R^K: a_1^{x_1} a_2^{x_2} \ldots a_c^{x_c} \mapsto (x_1,x_2, \ldots , x_c)$, where $x_i \in \R^{k_i}$. Recall that the Malcev coordinates also gave rise to isomorphisms $\gamma_i(G)/\gamma_{i+1}(G) \to \R^{k_i}$ and the induced endomorphism by $\varphi$ on
$\gamma_i(G)/\gamma_{i+1}(G)$ was represented by a matrix $M_i$. 
If we now let $\gamma g_i = a_1^{\alpha_1} a_2^{\alpha_2} \ldots a_c^{\alpha_c}$, it follows from Theorem~\ref{veel veeltermen} that we can view the map $\gamma g_i \varphi: G \to G$ as a map from $\R^K \to \R^K$ which maps 
$(x_1,x_2, \ldots, x_c)$ to 
\[ (\alpha_1+M_1x_1, \alpha_2 + M_2 x_2 + q_2(x_1), \ldots, \alpha_c + 
M_c x_c + q_c(x_1,x_2,\ldots,x_{c-1})),\]
where $q_i(x_1,x_2,\ldots, x_{i-1})$ is a polynomial depending on the variables $x_1, x_2, \ldots, x_{i-1}$ (and on $\alpha_1,\alpha_2, \ldots, \alpha_{i-1}$ and $M_1,M_2, \ldots, M_{i-1}$). Now let $x_0\in \R^K$ be the unique fixed point of this map. In order to show that the index of this fixed point is non-zero we can consider the Jacobian $J_{x_0}$ of this map at the point $x_0$. This Jacobian is of the form
\[    \begin{pmatrix}
    M_1         &0      & \hdots&       &  0
    \\ *       & M_2   &       &       &0
    \\ \vdots  & *     &\ddots &       &\vdots
    \\ \vdots  &       &       & \ddots&0
    \\ *       &*      &\hdots &   *   &M_c 
    \end{pmatrix}.\]
Note that also $\varphi_\ast$ is of this form, with exactly the same blocks $M_i$ on the diagonal (but other terms below the diagonal). As we assume that $1$ is not an eigenvalue of $\varphi_\ast$, we also have that $1$ is not an eigenvalue of $J_{x_0}$ and so we get that 
$$\det \Big( I - 
    \begin{pmatrix}
    M_1         &0      & \hdots&       &  0
    \\ *       & M_2   &       &       &0
    \\ \vdots  & *     &\ddots &       &\vdots
    \\ \vdots  &       &       & \ddots&0
    \\ *       &*      &\hdots &   *   &M_c 
    \end{pmatrix}
    \Big) \neq 0.$$
From page 12 of \cite{j},
it follows that the index equals the sign of the above determinant and hence equals $\pm 1$.
\end{proof}

\begin{remark}
    Note that in the lemma above we have that the $M_i$'s appearing on the block diagonal of the Jacobian are the same for all lift-factors. This means that the index of the fixed point classes are either all equal to 1 or all equal to -1.  
\end{remark}

\begin{theorem}\label{case 2}
    Let $f: \nilk{N}{G} \to D_n\left(\nilk{N}{G}\right)$
    be an irreducible affine $n-$valued map
    with basic lifting
    $\bar{f}^* = (g_1 \varphi, g_2 \varphi, \dots, g_n \varphi),$
    where $g_1, g_2, \dots, g_n \in G$ and $\varphi \in \text{End}(G).$
    If $\varphi_*$ does not have eigenvalue 1, then 
    $N(f)=  n  \Big|\det  \Big(I - \varphi_{*}\Big)\Big|.$
\end{theorem}

\begin{proof}
    Indeed, there are \[R(f)=  \sum_{i=1}^n \Big|\det  \Big(I - \varphi_{*}\Big)\Big|_\infty = n  \Big|\det  \Big(I - \varphi_{*}\Big)\Big| \]
    fixed point classes, each of index $\pm 1$, so they are all essential. It follows that in this case $N(f)=R(f)=n  \Big|\det  \Big(I - \varphi_{*}\Big)\Big| $.
\end{proof}

\subsection{The general case}
We are now ready to prove the main result of this paper, which generalizes the well known formula of Anosov for computing the Nielsen number of single-valued maps on nilmanifolds. 
\begin{theorem}\label{hoofdformule}
    Let $f: \nilk{N}{G} \to D_n\left(\nilk{N}{G}\right)$
    be an affine $n-$valued map
    with basic lifting
    $\bar{f}^* = (g_1 \varphi_1, g_2 \varphi_2, \dots, g_n \varphi_n),$
    where $g_1, g_2, \dots, g_n \in G$ and $\varphi_1,\varphi_2, \ldots , \varphi_n \in \text{End}(G).$ Then 
    \[R(f)=  \sum_{i=1}^n   \Big|\det  \Big(I - \varphi_{i*}\Big)\Big|_\infty 
    \mbox{ and }
    N(f)=  \sum_{i=1}^n   \Big|\det  \Big(I - \varphi_{i*}\Big)\Big|.\]
    When $R(f)<\infty$, then $R(f)=N(f)$.
\end{theorem}
\begin{proof}
The formula for the Reidemeister number was already obtained in Theorem~\ref{Formule Reidemeister}, so we only need to prove the formula for the Nielsen number.\\
By decomposing $f$ into its irreducible components $f_1, f_2, \ldots, f_r$ and by using the fact that $N(f) = N(f_1) + N(f_2) + \cdots + N(f_r)$, it suffices to prove this theorem in case $f$ is irreducible.
Note that in this case we have that $\varphi_1=\varphi_2=\cdots=\varphi_n$.

In case $\varphi_{1\ast}$ has eigenvalue 1, we know by Theorem~\ref{case 1} that $N(f)=0$, but then we also have that $\sum_{i=1}^n   \Big|\det  \Big(I - \varphi_{i*}\Big)\Big|=0$ since each term is 0. In the other case, Theorem~\ref{case 2} tells us that $N(f)= n  \Big|\det  \Big(I - \varphi_{1*}\Big)\Big|$ but this also equals $\sum_{i=1}^n   \Big|\det  \Big(I - \varphi_{i*}\Big)\Big|$ since all $\varphi_i$ are the same. 

The claim that $R(f)<\infty\Rightarrow R(f)=N(f)$ is now obvious.
\end{proof}
We illustrate the obtained formulas with some examples.
Section 4 of \cite{bdds} is devoted to the computation
of the Reidemeister and Nielsen number of the $3-$valued map
\[ f: T^2 \multimap T^2: p(t,s) \mapsto \left\{ 
p(\frac{t}{2},-s), p (\frac{t+1}{2},-s), p(-t,-s+\frac{1}{2}) \right\}\]
on the torus $T^2 = \nilk{\Z^2}{\R^2} ,$
where $p:\R^2 \to T^2$ is the  universal covering space.
The chosen basic lifting of $f$ is $\bar{f}^* = (\bar{f}^*_1, \bar{f}^*_2,  \bar{f}^*_3):\R^2 \to F_3(\R^2, \Z^2 ),$
with the maps $\bar{f}^\ast_i:\R^2 \to \R^2$ ($i=1,2,3$) defined by 
\begin{align*}
\bar{f}^\ast_1 (t,s)&= (\frac{t}{2},-s),
\\
\bar{f}^\ast_2 (t,s)&= (\frac{t+1}{2},-s),
\\
\bar{f}^\ast_3 (t,s)&= (-t,-s+\frac{1}{2}).
\end{align*}
The map $f$ has six fixed points:
$$
 p(0,0) , \;   p(0,\frac{1}{2}) , \;
  p(0,\frac{1}{4}) , \;   p(0,\frac{3}{4}) , \;
 p(\frac{1}{2},\frac{1}{4}) , \;  p(\frac{1}{2},\frac{3}{4}) .
$$
In \cite{bdds} we already found that $R(f)=6$ (which was in fact done by using formula (\ref{Som Reidemeister Getallen})),
and we also argued that the Nielsen number turns out to be the same in this case: $N(f) = 6.$
Now we can calculate those numbers easily again by using the above theorem.
With the notation from Theorem \ref{hoofdformule},
we denote the affine lift-factors as
\begin{align*}
     \bar{f}^\ast_1  = \left( g_1= \begin{pmatrix}
           0 \\
           0
         \end{pmatrix}, \varphi_1=
         \begin{pmatrix}
           \frac{1}{2} & 0\\
           0 & -1
         \end{pmatrix} \right)
         , \;
         \bar{f}^\ast_2  = 
         \left( g_2= \begin{pmatrix}
           \frac{1}{2} \\
           0
         \end{pmatrix},
         \varphi_2=
         \begin{pmatrix}
           \frac{1}{2} & 0\\
           0 & -1
         \end{pmatrix} \right)
         \\ \text{ and }
         \bar{f}^\ast_3  = \left( g_3= \begin{pmatrix}
           0 \\
           \frac{1}{2}
         \end{pmatrix},
         \varphi_3=
         \begin{pmatrix}
           -1 & 0\\
           0 & -1
         \end{pmatrix} \right).
    \end{align*}
    Since $G = \R^2,$
    it holds that $\varphi_i = \varphi_{i*} $ for $i = 1,2,3$
    and thus we get
    \begin{align*}
    \sum_{i=1}^3   &\Big|\det  \Big(I - \varphi_{i*}\Big)\Big|_\infty
    \\&= \Big|\det  \Big(I - 
        \begin{pmatrix}
           \frac{1}{2} & 0\\
           0 & -1
        \end{pmatrix}
    \Big)\Big|_\infty
    +
    \Big|\det  \Big(I - 
        \begin{pmatrix}
           \frac{1}{2} & 0\\
           0 & -1
         \end{pmatrix}
    \Big)\Big|_\infty
    +
    \Big|\det  \Big(I - 
        \begin{pmatrix}
           -1 & 0\\
           0 & -1
         \end{pmatrix}
    \Big)\Big|_\infty.
    \end{align*}
This results in  $1+1+4=6,$ which is indeed equal to the Reidemeister number $R(f)$ as computed before.
Since $R(f) < \infty,$ Theorem \ref{hoofdformule} implies that $N(f) = R(f) = 6,$
which again is in accordance with the Nielsen number computed in \cite{bdds}.
A slight modification of the map $f$ in the third component
gives us another example.
Consider the $3-$valued map
\[ g: T^2 \multimap T^2: p(t,s) \mapsto \left\{ 
p(\frac{t}{2},-s), p (\frac{t+1}{2},-s), p(t,-s+\frac{1}{2}) \right\}\]
with basic lifting $\bar{g}^* = (\bar{g}^*_1, \bar{g}^*_2,  \bar{g}^*_3):\R^2 \to F_3(\R^2, \Z^2 ),$
where the maps $\bar{g}^\ast_i:\R^2 \to \R^2$ ($i=1,2,3$) are defined by 
\begin{align*}
\bar{g}^\ast_1 (t,s)&= (\frac{t}{2},-s),
\\
\bar{g}^\ast_2 (t,s)&= (\frac{t+1}{2},-s),
\\
\bar{g}^\ast_3 (t,s)&= (t,-s+\frac{1}{2}).
\end{align*}
The map $g$ has infinitely much fixed points:
$$
 p(0,0), \;  p(0,\frac{1}{2}) \text{ and }
  p(t,\frac{1}{4}),  \;  p(t,\frac{3}{4}) 
 \text{ for } t \in \R.
$$
For the map $g$ we have
\begin{align*}
\varphi_1=
         \begin{pmatrix}
           \frac{1}{2} & 0\\
           0 & -1
         \end{pmatrix}, 
         \varphi_2=
         \begin{pmatrix}
           \frac{1}{2} & 0\\
           0 & -1
         \end{pmatrix}
         \text{ and }
         \varphi_3=
         \begin{pmatrix}
           1 & 0\\
           0 & -1
         \end{pmatrix} .
    \end{align*}
We compute
\begin{align*}
    \sum_{i=1}^3   &\Big|\det  \Big(I - \varphi_{i*}\Big)\Big|_\infty
    \\&= \Big|\det  \Big(I - 
        \begin{pmatrix}
           \frac{1}{2} & 0\\
           0 & -1
        \end{pmatrix}
    \Big)\Big|_\infty
    +
    \Big|\det  \Big(I - 
        \begin{pmatrix}
           \frac{1}{2} & 0\\
           0 & -1
         \end{pmatrix}
    \Big)\Big|_\infty
    +
    \Big|\det  \Big(I - 
        \begin{pmatrix}
           1 & 0\\
           0 & -1
         \end{pmatrix}
    \Big)\Big|_\infty,
    \end{align*}
    which implies that  $R(g)= |1|_\infty+|1|_\infty+|0|_\infty = 1 + 1 + \infty =\infty$.
        For the Nielsen number on the other hand, we get
  $N(g) = |1|+|1|+|0| = 2.$ 
  Note that the essential fixed point classes are $\{ p(0,0)\}$ and $\{p(0,\frac{1}{2}) \} $.


\begin{thebibliography}{00}

\bibitem{a} Anosov, D. V., \emph{Nielsen numbers of mappings of nil-manifolds}, Uspekhi Mat. Nauk {\bf 40}, 4(244), 133 - 134 (1985).

\bibitem{b} Brown, R., \emph{Fixed points of {$n$}-valued multimaps of the circle}, Bull. Pol. Acad. Sci. Math. {\bf 54}, 2, 153 - 162 (2006).

\bibitem{bdds} Brown, R., Deconinck, C., Dekimpe, K. and Staecker, P. C., \emph{Lifting classes for the fixed point theory of
$n$-valued maps}, Topol. Appl.
{\bf 274}, 107125 (2020).

\bibitem{bg} Brown, R. and Gon\c{c}alves, D., \emph{On the topology of
$n$-valued maps}, Adv. Fixed Point Theory {\bf 8}, 205 - 220 (2018).

\bibitem{c} Corwin, L. G. and Greenleaf F. P., \emph{Representations of nilpotent {L}ie groups and their applications. {P}art {I}}, Cambridge Studies in Advanced Mathematics {\bf 18}, Cambridge University Press (1990).

\bibitem{de} Dekimpe, K., \emph{A users' guide to infra-nilmanifolds and almost-{B}ieberbach
              groups}, Handbook of group actions. {V}ol. {III}, Adv. Lect. Math. {\bf 40}, Int. Press of Boston, 215 - 262 (2018).

\bibitem{d} Dekimpe, K., \emph{Almost-Bieberbach Groups: Affine and Polynomial Structures}, Springer (1996).

\bibitem{dp} Der\'{e}, J. and Pengitore, M., \emph{Effective twisted conjugacy separability of nilpotent groups}, Math. Z. {\bf 292}, 3 - 4, 763 - 790 (2019).

\bibitem{fh} Fadell, E. and Husseini, S., \emph{On a theorem of {A}nosov on {N}ielsen numbers for nilmanifolds}, Nonlinear functional analysis and its applications ({M}aratea, 1985), NATO Adv. Sci. Inst. Ser. C: Math. Phys. Sci. {\bf 173}, Reidel, 47 - 53 (1986).

\bibitem{fgw} Fel'shtyn, A., Gon\c{c}alves, D. and Wong, P., \emph{Twisted conjugacy classes for polyfree groups}, Comm. Algebra {\bf 42}, 1, 130 - 138 (2014).

\bibitem{gw} Gon\c{c}alves, D. and Wong, P., \emph{Twisted conjugacy classes in nilpotent groups}, J. Reine Angew. Math. {\bf 633}, 11 - 27 (2009).

\bibitem{jm} Jezierski, J. and Marzantowicz, W., \emph{Homotopy methods in topological fixed and periodic points theory}, Topological Fixed Point Theory and Its Applications {\bf 3}, Springer (2006).

\bibitem{ji} Jiang, B., \emph{A primer of {N}ielsen fixed point theory}, Handbook of topological fixed point theory, Springer, 617 - 645 (2005).

\bibitem{j} Jiang, B., \emph{Lectures on Nielsen Fixed Point Theory}, Contemp. Math. {\bf 14} (1983).

\bibitem{k} Kiang, T., \emph{The theory of fixed point classes}, Springer-Verlag (1989).

\bibitem{l} Lee, K. B., \emph{Maps on infra-nilmanifolds}, Pacific J. Math. { \bf 168}, 1, 157 - 166 (1995).

\bibitem{m} Malcev, A. I., \emph{On a class of homogeneous spaces}, Amer. Math. Soc. Translation {\bf 39} (1951).

\bibitem{ms} Mitra, O. and Sankaran, P., \emph{Twisted conjugacy in {${\rm GL}_2$} and {${\rm SL}_2$} over polynomial algebras over finite fields}, Geom. Dedicata {\bf 216}, 2 (2022).

\bibitem{s} Schirmer, H., \emph{An index and Nielsen number for 
$n$-valued multifunctions}, Fund. Math. {\bf 121}, 201 - 219 (1984).

\bibitem{sch} Schirmer, H., \emph{Fix-finite approximation of {$n$}-valued multifunctions}, Fund. Math. {\bf 121}, 1, 73 - 80 (1984).

\bibitem{se} Segal, D., \emph{Polycyclic groups}, Cambridge Tracts in Mathematics {\bf 82}, Cambridge University Press (1983).

\bibitem{sen} Senden, P., \emph{Twisted conjugacy in direct products of groups}, Comm. Algebra {\bf 49}, 12, 5402- 5422 (2021).

\bibitem{st} Staecker, P. C., \emph{Partitions of $n-$valued maps}, 	arXiv:2101.09326 [math.GN] (2021).

\bibitem{w} Wecken, F., \emph{Fixpunktklassen. {III}. {M}indestzahlen von {F}ixpunkten}, Math. Ann. {\bf 118}, 544 - 577 (1942).

\bibitem{x} Xicot\'encatl, M., \emph{Orbit configuration spaces}, 
Contemp. Math. {\bf 621}, 113 - 132 (2014).


\end{thebibliography}
\end{document}